\documentclass[11pt, leqno, lettersize]{amsart}
\numberwithin{equation}{section}

\usepackage{amsmath,amsthm,amsfonts,amssymb,graphicx}
\usepackage[backref=page]{hyperref}
\hypersetup{colorlinks=true,citecolor=blue,linkcolor=blue,urlcolor=blue,pdfstartview=FitH,
pdfauthor=Blomer/Goldmakher/Louvel,pdftitle=$L$-functions with $n$-th order twists}

\theoremstyle{plain}
\newtheorem{thm}{Theorem}[section]
\newtheorem{prop}[thm]{Proposition}
\newtheorem{cor}[thm]{Corollary}
\newtheorem{lem}[thm]{Lemma}

\theoremstyle{definition}

\theoremstyle{remark}

\theoremstyle{plain} \newtheorem*{GS2.1}{Theorem A}

\renewcommand{\geq}{\geqslant}
\renewcommand{\leq}{\leqslant}
\renewcommand{\mod}[1]{{\ifmmode\text{\rm\ (mod~$#1$)}\else\discretionary{}{}{\hbox{ }}\rm(mod~$#1$)\fi}}
\renewcommand{\Re}{\textup{Re }} 
\renewcommand{\Im}{\textup{Im }}

\newsymbol\dnd 232D

\newcommand{\sumstar}[1]{\sideset{}{^*}\sum_{#1}}  
\newcommand{\sumstarline}[1]{\left.\sum\right.^*_{#1}}  

\newcommand{\I}{\mathcal{I}}
\newcommand{\oh}{\mathcal{O}}
\newcommand{\n}{\mathcal{N}}

\newcommand{\N}{\mathbb{N}}
\newcommand{\Z}{\mathbb{Z}}
\newcommand{\Q}{\mathbb{Q}}
\newcommand{\R}{\mathbb{R}}
\newcommand{\C}{\mathbb{C}}

\newcommand{\A}{\mathfrak{a}}
\newcommand{\B}{\mathfrak{b}}
\newcommand{\frakc}{\mathfrak{c}}
\newcommand{\m}{\mathfrak{m}}
\newcommand{\f}{\mathfrak{f}}
\newcommand{\pr}{\mathfrak{p}}

\begin{document}

\title{$L$-functions with $n$-th order twists}

\author{Valentin Blomer}
\address{Mathematisches Institut, Bunsenstr. 3-5, 37073 G\"ottingen, Germany} \email{blomer@uni-math.gwdg.de}
\author{Leo Goldmakher}
\address{Department of Mathematics, University of Toronto, Toronto, ON, Canada} \email{leo.goldmakher@utoronto.ca}
\author{Beno\^it Louvel}
\address{Mathematisches Institut, Bunsenstr. 3-5, 37073 G\"ottingen, Germany} \email{blouvel@uni-math.gwdg.de}

\thanks{First author supported  by the Volkswagen Foundation and an ERC Starting Grant. Second author supported in part by an NSERC grant. Third author supported by the Volkswagen Foundation.}


\begin{abstract}
Let $K$ be a number field containing the $n$-th roots of unity for some $n \geq 3$. We prove a uniform subconvexity result for a family of double Dirichlet series built out of central values of Hecke $L$-functions of $n$-th order characters of $K$. The main new ingredient, possibly of independent interest, is a large sieve for $n$-th order characters. As further applications of this tool, we derive several results concerning $L(s,\chi)$ with $\chi$ an $n$-th order Hecke character: an estimate of the second moment on the critical line, a non-vanishing result at the central point, and a zero-density theorem.
\end{abstract}


\maketitle

\section{Introduction} \label{sect:Intro}

Analytic properties of $L$-functions in the critical strip often reflect subtle arithmetic properties of the number theoretic object encoded in the coefficients. In particular the central value has received much attention. If one has a suitable family of $L$-functions, their central values can again be encoded into an $L$-function, and one obtains a multiple Dirichlet series. Typically such multiple Dirichlet series no longer have Euler products, but they may have other interesting structure such as a non-trivial group of functional equations.\\

One example of such a double Dirichlet series was investigated by Friedberg, Hoffstein, and Lieman in \cite{FHL}. Their construction (which we review below in detail) is somewhat technical, but the idea is straightforward. Fix $n \geq 3$, let $K$ be a (totally imaginary) number field of degree $d \, (\geq \phi(n))$ over $\Bbb{Q}$ containing the $n$-th roots of unity, and denote by $\oh$ the ring of integers of $K$. We fix a set $S$ of (bad) finite places of $K$ and a suitable ideal $\mathfrak{c}$ depending only on $K$. There is a family of $n$-th order Hecke characters $\chi_\A$ of $K$ indexed by integral ideals $\A \subseteq \oh$ coprime to $S$, which are a natural analogue of the family of quadratic Dirichlet characters
$\chi_{_{D}} = \left(\frac{D}{\cdot}\right)$. If $\A$ is squarefree, then the conductor of $\chi_{\A}$ is an ideal lying between $\A$ and $\mathfrak{c} \A$. We will discuss properties of these characters in detail in Section \ref{sect:char}. While the quadratic characters $\chi_{D}$ have been studied extensively,   little is known about their higher order relatives.\\

One can now construct a double Dirichlet series $Z(s,w)$ which (in the region of absolute convergence) looks roughly like $  \sum_\A L(s,\chi_\A) (\n\A)^{-w}$ where $\n$ denotes the norm of an ideal.
This double Dirichlet series has two obvious functional equations: one arising from the functional equation of $L(s,\chi_\A)$, and another from reciprocity of the character $\chi_\A$. These two functional equations generate many others, including one which takes $(s,w) \leftrightarrow (1-s,1-w)$. Our first  result is a subconvexity result for the function $Z(1/2, w)$ on the critical line $\Re w = 1/2$ (Theorem \ref{thm:sub}). In fact, we prove rather more: viewed as a function of two variables $s$, $w$, we deduce a \emph{uniform} subconvexity bound for $Z(s, w)$ on the plane
$\Re s = \Re w = 1/2$ (Theorem \ref{thm:sub1}). The main new ingredient in our proofs of these subconvexity results is Theorem \ref{thm:thm3}, a large sieve over the family of $n$-th order Hecke characters $\chi_\A$.
To further illustrate the utility of this tool, we deduce several results concerning the behavior of the Hecke $L$-function $L(s,\chi_\A)$: an estimate of the second moment on the critical line (Corollary \ref{cor:kor2}), a non-vanishing result at the central point (Corollary \ref{cor:kor2a}), and a zero-density theorem for zeros off the critical line (Corollary \ref{cor:kor3}).\\

As is well-known in the theory of multiple Dirichlet series, in order to obtain the desired functional equations one needs to modify the objects in question by attaching some correction factors.
With this in mind, let $H_{\frakc}$ be the ray class group modulo $\frakc$, $R_{\frakc} := H_{\frakc} \otimes \Z/n\Z$, and let $\psi$ be a (unitary) id\`ele class character of conductor dividing $\frakc$ and order dividing $n$, that is, a character of $R_{\frakc}$. Let $S_{\mathfrak{a}}$ denote the set of prime ideals dividing  the conductor of $\chi_{\A}$.
For $s \in \C$ with sufficiently large real part define
\begin{equation}
    \label{eq:defL*}
    L^*(s, \psi, \A) = L_{S \cup S_{\A}}(s, \psi \chi_{\A}) \, \frak{A}(s, \psi, \A)
\end{equation}
where the subscript $S \cup S_{\A}$ indicates that the Euler factors at places in $S\cup S_{\A}$ have been removed, and where
\begin{equation}
    \label{eq:defa}
    \frak{A}(s, \psi, \A_1 \A_2^n) :=
    \sum_{\B_1 \B_2 \B_3 = \A_2}
    \frac{\chi_{\A}(\B_3)\mu(\B_3)\psi(\B_3)}{\n\B_1^{ns - n + 1} \n\B_3^s}
    \ll
    \n\A_2^{\max(0, n(1 - \Re s) - 1) + \varepsilon}
\end{equation}
(with $\A_1$ $n$-th power free) is a suitable correction factor at the ramified primes.  
For two characters $\psi, \psi'$ of $R_{\frakc}$ and $s, w$ with sufficiently large real part we define
\begin{equation}
    \label{eq:defz1}
    Z_1(s, w; \psi, \psi') := \sum_{\A \in \I(S)} \frac{L^{\ast}(s, \psi, \A) \psi'(\A)}{\n\A^w}
\end{equation}
where $\I(S)$ is the set of nonzero integral ideals coprime to $S$.
This function has been studied in detail in \cite{FHL}; we will review the theory in section
\ref{sect:DoubleDirichlet}. For the moment we just mention that this is the right object to consider with a clean set of functional equations. It has been conjectured in \cite{BBCFH} that $Z_1$ together with its companion function $Z_2$ defined in \eqref{eq:defz2} below is a residue of a Weyl group multiple Dirichlet series  of type $A_3^{(3)}$. The conjecture has been verified for $n=3$ in \cite{BB}. \\

Specializing to the central value $s=1/2$, the function $Z_1$ has a functional equation in $w$. The completed $L$-function looks roughly like $\Gamma(nw)^{d/2} Z_1(1/2, w; \psi, \psi')$
and is essentially invariant under $w \leftrightarrow 1-w$. A standard convexity argument shows
\[
    Z_1(1/2, w, \psi, \psi') \ll |w|^{\frac{nd}{4} + \varepsilon}, \quad \Re w = 1/2.
\]
Our first result is a power saving of this estimate and establishes a subconvexity result for this $L$-function.
\begin{thm}
    \label{thm:sub}
    With the notation developed above we have
    \[
    Z_1(1/2, w, \psi, \psi') \ll |w|^{\frac{nd}{4} - \frac{d}{12} + \varepsilon},
    \quad \Re w = 1/2.
    \]
    for any $\varepsilon > 0$.
\end{thm}

We emphasize that here and henceforth all implied constants may depend on $n$, $K$ (including $S$, $\mathfrak{c}$ etc.) and $\varepsilon$ even if this is not explicitly mentioned. With more work this bound can be improved; our aim was to provide a subconvexity result with as little technical work as possible. The functional equation for $Z_1(1/2, w; \psi, \psi')$ and the bound of Theorem \ref{thm:sub} are shadows of the general situation when $Z_1(s, w; \psi, \psi')$ is treated as a function of two complex variables $s$ and $w$. For clarity of exposition we give a simplified description of the situation and postpone precise statements to later sections. Let us assume for simplicity that $\psi, \psi'$ are trivial and drop them from the notation. In the region of absolute convergence, the double Dirichlet series \eqref{eq:defz1} looks roughly like
$\displaystyle \sum_{\A, \B} \chi_{\A}(\B) \n\A^{-s} \n\B^{-w}$. Applying reciprocity, we see that
\begin{equation}
    \label{eq:f1}
    Z_1(s, w) \approx Z_1(w, s).
\end{equation}
Alternatively, we can apply the functional equation in the numerator of \eqref{eq:defz1}. This will introduce $n$-th order Gau{\ss} sums, and therefore we need an auxiliary double Dirichlet series of the shape
\[
    Z_2(s, w) := \sum_{\A} \frac{D(w, \A)}{\n\A^s}, \quad
    D(w, \A) \approx \sum_{\B} \frac{\epsilon(\chi_{\B})\overline{\chi_{\B}}(\A)}{\n\B^w}
\]
where $\epsilon(\chi_{\B})$ is the normalized Gau{\ss} sum (of absolute value 1) appearing as the root number in the functional equation of $L(s, \chi_{\B})$. The function $D(w, \A)$ also has a functional equation, as it can essentially be realized as the Fourier coefficient of an Eisenstein series on the $n$-fold cover of $GL(2)$ over $K$. The theory has been developed in detail by Kazhdan and Patterson in \cite{KP}. We note already at this point an interesting asymmetry between $D(w, \A)$ and $L^*(s, \textbf{1}, \A)$. The function $L^*(s, \textbf{1}, \A)$ is a standard $GL(1)$ $L$-function over a number field with $d/2$ complex places,
hence its archimedean local factor is essentially $\Gamma(s)^{d/2}$. On the other hand, $D(w, \A)$ is a more complicated object, and its archimedean local factor is essentially $\Gamma(w)^{(n-1)d/2}$.

We now give a simplified description of the functional equations of $Z_1$ and $Z_2$. The completed versions are very roughly
\begin{equation}
    \label{eq:complete}
    \begin{split}
        \Xi_1(s, w) & \approx \bigg(\Gamma(s)\Gamma\Big((n-1)(s+w-1)\Big)\Gamma(w)\bigg)^{d/2}  Z_1(s, w),\\
        \Xi_2(s, w) & \approx
        \bigg(\Gamma(s) \Gamma\Big((n-1)(w-1/2)\Big) \Gamma(s+w)\bigg)^{d/2} Z_2(s, w).
    \end{split}
\end{equation}
Then in addition to \eqref{eq:f1} one has
\begin{equation}
    \label{eq:f1+}
    \begin{split}
        &\Xi_2(1 - s, w + s - 1 / 2) \approx \Xi_1(s, w),\\
        & \Xi_2(s + w - 1/2, 1 - w) \approx \Xi_2(s, w).
    \end{split}
\end{equation}
These functional equations can be iterated and generate a finite group. In particular, we obtain
\begin{equation}
    \label{eq:f1++}
    \Xi_1 (s, w) \approx \Xi_1(1 - s, 1 - w), \quad \Xi_2(s, w) \approx \Xi_2(1 - s, 1 - w).
\end{equation}

There is no obvious concept of analytic conductor or convexity bound in the context of multiple Dirichlet series, but in view of the functional equation \eqref{eq:f1++}  and \eqref{eq:complete} it seems reasonable to measure the complexity of $Z_1(s, w)$ for $\Re s = \Re w = 1/2$ in terms of the quantity
\begin{equation}
    \label{eq:defC0}
    C :=  (|s| |s+w|^{n-1} |w|)^{d} .
\end{equation}
We expect a ``convexity" bound
\begin{equation}
    \label{eq:convexity}
    Z_1(s, w) \ll C^{1/4+\varepsilon} , \quad \Re s = \Re w = 1/2 .
\end{equation}
These notions agree with the classical ones if we specialize to the one variable case
$s = 1/2$ or $w = 1/2$. For illustrative purposes, we briefly sketch in Appendix I  a convexity argument which would yield the bound \eqref{eq:convexity}.
We are now ready to state the following result, which improves on the convexity bound \eqref{eq:convexity} and generalizes Theorem \ref{thm:sub}.
\begin{thm}
    \label{thm:sub1}
    With $C$ as in \eqref{eq:defC0} we have
    \[
    Z_1(s, w; \psi, \psi') \ll C^{\frac{1}{4} - \frac{1}{12(n+1)} + \varepsilon}
    \]
    for $\Re s = \Re w = 1/2$ and any $\varepsilon > 0$.
\end{thm}
\noindent
This provides a \emph{uniform} subconvexity result in all ranges of $\Im s$ and $\Im w$, even in ``exceptional" ranges like $\Im s = -\Im w$ where the conductor drops considerably. Again Theorem \ref{thm:sub1} may be improved with somewhat more work. The reader may recognize certain features in the proof of Theorem \ref{thm:sub1} from \cite{Bl}, where the case $n=2$ is treated.\\

One of the main new ingredients of the present paper is a tool of independent interest, a large sieve inequality for $n$-th order characters. Such results for $n=2$ and $n=3$ have been obtained by Heath-Brown \cite{HB, HB1}.\footnote{Baier and Young \cite{BY} have developed a variant of Heath-Brown's cubic large sieve over $\Q$, and used this to bound the second moment of $L$-functions associated to cubic characters.}
Recently, in
\cite[Theorem 1.1]{GZ}, a weaker analogue of Heath-Brown's results was proved for the case $n=4$. The following theorem, which is valid for all\footnote{for $n=2$ better bounds are proved in \cite{GL}.} $n \geq 3$, is of the same quality as \cite{HB1} and improves in particular the main result of \cite{GZ}.
\begin{thm}
    \label{thm:thm3}
    Let $M, N \geq 1/2$, and let $\lambda_{\A} \in \C$ be a sequence of complex numbers indexed by $\A \in \I(S)$. Then
    \[
    \sumstar{   \n\A \leq M} \,
    \Bigl| \sumstar{ \n\B\leq N} \lambda_{\B} \chi_{\A}(\B) \Bigr|^2
    \ll
    (MN)^\varepsilon (M+N +(MN)^{2/3})
    \sumstar{  \n\B \leq N} |\lambda_{\B} |^2
    \]
    for any $\varepsilon > 0$. Here and henceforth $\sumstarline{}$ indicates that the summation is restricted to squarefree ideals coprime to $S$.
\end{thm}
As in \cite{HB1}, the term $(MN)^{2/3}$ is not optimal and can most likely be removed. Nevertheless, Theorem \ref{thm:thm3} can be useful in various situations, as the following applications demonstrate.

\begin{cor}
    \label{cor:kor2}
    For $N \geq 1$, $t \in \R$ and $\varepsilon > 0$, one has with the notation developed so far
    \[
    \sum_{\substack{\A \in \I(S) \\ \n\A \leq N}} |L(1/2+it, \chi_{\A})|^2
    \ll
    \left( N (1+|t|)^{d/2}\right)^{1+\varepsilon} .
    \]
\end{cor}
\noindent This as  strong as Lindel\"of with respect to $\text{cond}(\chi_{\A})$ together with the convexity bound with respect to $t$. For fixed $t$, this result is not new. In fact, an asymptotic formula for a smoothed second moment has been established in \cite[Theorem 4.1]{Di}. In principle this can certainly be made uniform in $t$, but the outcome is not clear a priori. Our method is comparatively short and elementary and gives a uniform upper bound. 

\begin{cor}
    \label{cor:kor2a}
    For $N \geq 1$ sufficiently large and $\varepsilon > 0$ we have 
    \[
    \#\{ \A \in \I(S)  \mid  \n\A \leq N, \: L(1/2, \chi_{\A}) \neq 0\}
    \gg
    N^{1-\varepsilon} .
    \]
\end{cor}
\noindent For $n=3$ this was proved in \cite{L}. 

\begin{cor}
    \label{cor:kor3}
    For $1/2 < \sigma \leq 1$,  $T \geq 1$ and a squarefree integral ideal $\A$ coprime to $S$ let $N(\sigma, T, \A)$ be the number of zeros $\rho = \beta + i \gamma$ of
    $L(s, \chi_\A)$ in the rectangle $\sigma \leq \beta \leq 1$, $|\gamma| \leq T$. Then
    \[
    \sumstar{  \n\A \leq N} N(\sigma, T, \A)
    \ll
    N^{g(\sigma)} T^{1+ d\frac{1-\sigma}{3-2\sigma}} (NT)^{\varepsilon}, \quad
    g(\sigma) :=
    \begin{cases}
       \frac{2(10\sigma-7)(1-\sigma)}{24\sigma - 12\sigma^2 - 11}, & 5/6 < \sigma \leq 1,\\[0.2cm]
      \frac{8(1-\sigma)}{7-6\sigma}, & 1/2 <  \sigma \leq 5/6 .
    \end{cases}
       \]
\end{cor}
\begin{center}
\includegraphics[width = 6cm]{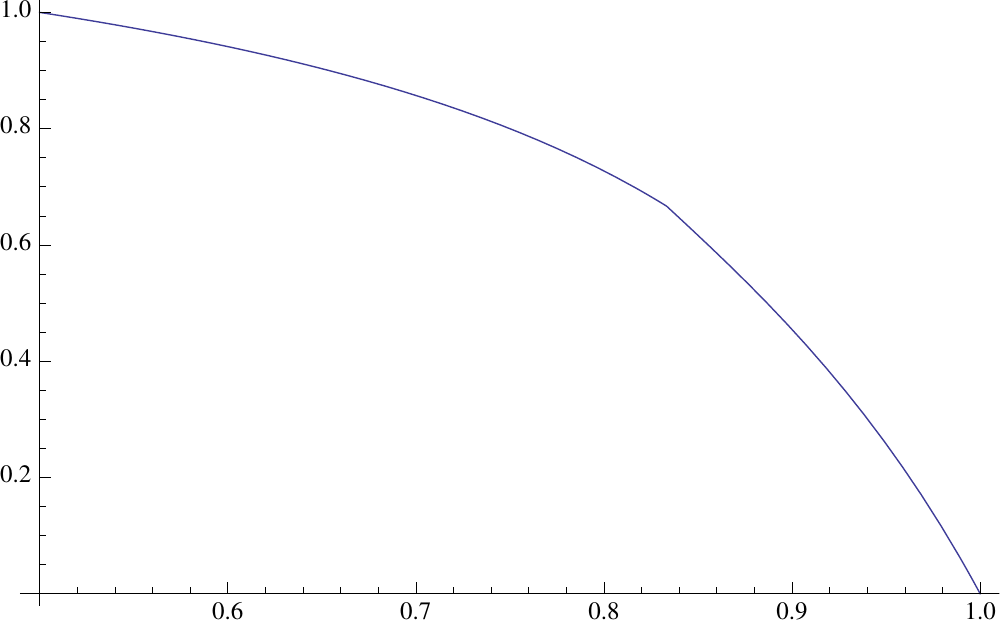}\\
Plot of $g(\sigma)$ with $1/2 \leq \sigma \leq 1$.
\end{center}
$\\$

This is relatively weak in the $T$ aspect (and we have not optimized the $T$-exponent), but non-trivial in the $N$-aspect for any $\sigma > 1/2$. In particular, this shows that for any $1/2 < \sigma \leq 1$ there are at most $N^{g(\sigma)+\varepsilon}$ characters $\chi_\A$ with
$\n\A \leq N$ such that $L(s, \chi_\A)$ has a \emph{real} zero in the segment $[\sigma, 1]$. A similar zero density result in the special case $n=3$ that is weaker in the $N$-aspect and stronger in the $T$-aspect was proved in \cite{X}, but the proof appears to be incorrect\footnote{The argument between (9) and (10) in \cite{X} would prove the Lindel\"of hypothesis in the $t$-aspect.}.\\

By a standard procedure (see e.g.\ \cite{Ro, X}), Theorem \ref{thm:thm3} and Corollary \ref{cor:kor3} can be used to determine completely the statistical properties of
$L(1, \chi_\A)$ and in particular compute all its moments. We leave this and further applications of the $n$-th order large sieve to the interested reader.


\noindent\\
\textbf{Acknowledgments.} The second author is grateful to the mathematics department at the University of G\"ottingen for its hospitality. The third author thanks the Centre Interfacultaire Bernoulli at the Ecole Polytechnique F\'ed\'erale de Lausanne, for its hospitality during the semester on Group Actions in Number Theory.

\section{The $n$-th order residue symbol}\label{sect:char}


Starting from the $n$-th order power residue symbol $(a/.)$ attached to the field extension $K(a^{1/n})/K$, Fisher and Friedberg \cite{FF} constructed\footnote{originally in the function
field case and $n=2$, but the number field case and general $n$ is similar, see  \cite{FHL}} for each ideal $\mathfrak{a} \in \mathcal{I}(S)$ a Hecke character $\chi_{\mathfrak{a}}$  that
generalizes the power residue symbol $(a/.)$, in the sense that for 
$a \equiv 1$ (mod $\mathfrak{c}$) one has $\chi_{(a)} = \chi_a$.
Here we recall that we have chosen once and for all a sufficiently small ideal $\mathfrak{c}$  whose construction depends only on $n$ and the field $K$, and we also recall  the definition of the finite group $R_{\frakc} = H_{\frakc} \otimes \Bbb{Z}/n\Bbb{Z}$ from the previous section.  For an ideal $\A \in \mathcal{I}(S)$ we shall denote by $[\A]$ the image of $\A$ in $R_{\mathfrak{c}}$.
For the convenience of the reader we shall give precise definitions in Appendix II, and quote here only the relevant facts which we need for the rest of the paper; the interested reader is also referred to the detailed accounts in \cite{FF, FHL, GL}. 
We point out that we follow the convention in these references to   always mean the underlying \emph{primitive} Hecke character when we write  $\chi_{\mathfrak{a}}$.  In other words, $\chi_{\mathfrak{a}}$ is by definition a primitive Hecke character.
For example, $\chi_{\mathfrak{a}^n}$  is the arithmetic function on integral ideals that returns the value 1 for \emph{all} non-zero integral
ideals $\mathfrak{b}$, whether or not they are coprime to $\mathfrak{a}$.  On the other hand, a product $\chi_{\mathfrak{a}} \chi_{\mathfrak{b}}$ is the product of the two arithmetic functions $\chi_{\A}$, $\chi_{\B}$ defined on ideals and not necessarily\footnote{This differs form the convention in \cite{GL}.} a primitive character (cf.\ also Lemma \ref{conductor} and  the beginning of Section \ref{sect:large-sieve}). The various statements of the following lemma are proved in \cite[Appendix\,B]{GL}.

\begin{lem}\label{lemma21}  For any ideal $\mathfrak{a} \in \mathcal{I}(S)$ the character $\chi_{\mathfrak{a}}$ is a primitive $n$-th order Hecke character of trivial infinite type. Let $\mathfrak{a}, \mathfrak{b} \in \mathcal{I}(S)$. Then
\begin{displaymath}
\chi_{\mathfrak{a}} \chi_{\mathfrak{b}} = \chi_{\mathfrak{a}\mathfrak{b}} \, \text{ if } \mathfrak{a}\mathfrak{b}  \text{ is $n- $th power free}\qquad \textrm{and} \qquad \chi_{\mathfrak{a}\mathfrak{b}^n} = \chi_{\mathfrak{a}}.
\end{displaymath}
If $\A = \A_1\A_2^n$ with $\A_1$ $n$-th power free, then the conductor of $\chi_{\A}$ is an ideal lying between $\A_0$ and $\frakc \A_0$, where $\A_0$ is the squarefree kernel of $\A_1$.
Finally a reciprocity formula holds in the sense that the value of  $\chi_{\mathfrak{a}}(\mathfrak{b}) \chi_{\mathfrak{b}}(\mathfrak{a})^{-1}$ depends only on the classes $[\mathfrak{a}], [\mathfrak{b}] \in R_{\mathfrak{c}}$.
\end{lem}


From the Chinese Remainder Theorem we conclude that any finite order Hecke character $\chi$ modulo $\mathfrak{f} = \mathfrak{f}_1\mathfrak{f}_2$ with $\mathfrak{f}_1, \mathfrak{f}_2$ coprime decomposes as $\chi = \chi_1\chi_2$ where $\chi_j$ is a Hecke character modulo $\mathfrak{f}_j$. This decomposition is unique up to multiplication by a character of the class group $\text{Cl}$ of $K$, and $\chi$ is primitive if and only if $\chi_1$ and $\chi_2$ are primitive. To see this,  denote by $G(\mathfrak{f})$ the set of Hecke characters of conductor dividing $\mathfrak{f}$ which are unramified at infinity. Let $\mathcal{O}^{\times}_{\mathfrak{f}}$ denote the image of the units $\mathcal{O}^{\times}$ in $(\mathcal{O}/\mathfrak{f})^{\times}$. Then
\begin{displaymath}
  G(\mathfrak{f})\cong \{(\chi, \chi_{\text{Cl}}) \in \widehat{(\mathcal{O}/\mathfrak{f})^{\times}} \times \widehat{\text{Cl}} \mid \chi(\mathcal{O}^{\times}_{\mathfrak{f}}) = 1\},
\end{displaymath}
and by a simple counting argument it is easy to see that the map $G(\mathfrak{f}_1) \times G(\mathfrak{f}_2) \rightarrow G(\mathfrak{f})$ is $h$-to-1, in particular surjective.\\

We need this observation in the following situation: if $\B \in \mathcal{I}(S)$ is squarefree, the character $\chi_{\B}$ factorizes as
\begin{equation}\label{decompos}
  \chi_{\B} = \chi^{(\B)} \psi_{(\B)},
\end{equation}
say, where $\chi^{(\B)}$ is a primitive Hecke character of conductor $\B$ and $\psi_{(\B)}$ is a Hecke character of conductor dividing $\frakc$. We choose once and for all such a decomposition but fixing a choice of class group characters. \\



The $L$-function attached to any primitive ray class character $\chi$ of conductor $\mathfrak{f}$ satisfies the functional equation
\begin{equation*}
\left(\frac{|d_K|}{(2\pi)^d}\right)^{s/2} \Gamma(s)^{d/2} L(s,\chi) = \epsilon(\chi) (\n \f)^{1/2-s}  \left(\frac{|d_K|}{(2\pi)^d}\right)^{(1-s)/2} \Gamma(1-s)^{d/2}L(1-s,\bar{\chi})
\end{equation*}
where $d_K$ denotes the discriminant of $K$, and $\epsilon(\chi)$ is the epsilon-factor. 
For a fixed smooth function $W$ with compact support in $(0, \infty)$, let $\widehat{W}(s)$ denote its Mellin transform and let
\begin{displaymath}
  \dot{W}(x) = \frac{1}{2\pi i} \int_{(c)} \widehat{W}(1-s) \left(\frac{(2\pi)^s\Gamma(s)}{\Gamma(1-s)}\right)^{d/2} |d_K|^{-s/2} x^{-s} ds,
\end{displaymath}
for $c > 0$. A contour shift shows 
\begin{equation}\label{eq:B3-B2:2}
\dot{W}(x) \ll  |x|^{-A},
\end{equation}
for any $A > 0$. A standard argument now yields  (cf.\ \cite[Lemma A.2]{GL})
\begin{lem}\label{poisson} Let $M > 0$. With the notation as above, we have
\begin{displaymath}
\sum_{\A } W\left(\frac{\n\A}{M}\right) \chi(\A) = \frac{M \epsilon(\chi)}{\sqrt{\n \f}} \sum_{\A} \dot{W}\left(\frac{M \n\A}{\n \f}\right) \overline{\chi}(\A).
\end{displaymath}
\end{lem}
We will apply this later to a character of the type $\chi_{\B_1}^j\bar{\chi}^j_{\B_2}$ for $j \in \Bbb{N}$. Therefore we state the following auxiliary result:
\begin{lem}\label{conductor}
Let $\B_1, \B_2 \in \mathcal{I}(S)$ be coprime and squarefree, and assume that $[\B_1] = [\B_2] = \mathcal{C} \in R_{\mathfrak{c}}$. Let $1 \leq j < n$. Let $\chi$ be the primitive Hecke character that induces  $\chi^j_{\B_1}\bar{\chi}^j_{\B_2}$. Then $\chi = (\chi^{(\B_1)})^j (\bar{\chi}^{(\B_2)})^j$ with the notation as in \eqref{decompos}. In particular, $\chi$ has conductor $\B_1\B_2$ and root number
\begin{equation}\label{gauss}
  \epsilon(\chi) = \eta \epsilon( (\chi^{(\B_1)})^j) \epsilon((\bar{\chi}^{(\B_2)})^j) \bar{\psi}_{(\B_2)}(\B_2)^j \psi_{(\B_1)}(\B_1)^j
\end{equation}
where $\eta $ is a constant depending only on $\mathcal{C}$.
\end{lem}

\begin{proof}
The fact that $\chi$ has conductor $\B_1\B_2$ is proved in \cite[Appendix\,B]{GL}.
It follows that
\begin{equation}\label{factorize}
  \chi = (\chi^{(\B_1)})^j  (\bar{\chi}^{(\B_2)})^j,
\end{equation}
and by the Chinese Remainder Theorem and \eqref{decompos} we obtain
\begin{displaymath}
\begin{split}
  \epsilon(\chi)& =  \epsilon( (\chi^{(\B_1)})^j) \epsilon((\bar{\chi}^{(\B_2)})^j) \chi^{(\B_1)}(\B_2)^j \bar{\chi}^{(\B_2)}(\B_1)^j \\
  &= \epsilon( (\chi^{(\B_1)})^j) \epsilon((\bar{\chi}^{(\B_2)})^j) \chi_{\B_1}(\B_2)^j \bar{\chi}_{\B_2}(\B_1)^j \bar{\psi}_{(\B_1)}(\B_2)^j \psi_{(\B_2)}(\B_1)^j.
\end{split}
\end{displaymath}
Now \eqref{factorize} implies that $\psi_{(\B_1)}^j\bar{\psi}_{(\B_2)}^j$ is trivial on ideals coprime to $S$. Hence $\bar{\psi}_{(\B_1)}(\B_2)^j \psi_{(\B_2)}(\B_1)^j = \bar{\psi}_{(\B_2)}(\B_2)^j \psi_{(\B_1)}(\B_1)^j $, and by reciprocity $\eta :=  \chi_{\B_1}(\B_2)^j \bar{\chi}_{\B_2}(\B_1)^j $ depends only on $\mathcal{C}$. This completes the proof of \eqref{gauss}.
\end{proof}


\section{A large sieve inequality for $n$-th order characters}
\label{sect:large-sieve}

In this section we prove Theorem \ref{thm:thm3}. The proof is inspired by the work of Heath-Brown \cite[Theorem 2]{HB1}. Some extra complications arise from the fact that the character $\chi_{\A}$ and its reciprocity law are not as explicit as in the cases $n=2, 3$. 
For simplicity of notation we shall write 
$\A\sim M$ to mean $M < \n\A \le 2M$. For $\A,\B\in \I(S)$ square-free, coprime and in the same class $\mathcal{C}$, we denote by $\chi_{\A,\B} = \chi^{(\A)} \overline{\chi}^{(\B)}$ the primitive character inducing $\chi_\A\overline{\chi}_\B$ as in Lemma \ref{conductor}. 

\subsection{Preliminaries}
\label{subsect:companion}

As in the previous section let $W$ be a fixed smooth function which is non-zero on $[1, 2]$ and has compact support in $(0, \infty)$. For any sequence of complex numbers $\lambda=(\lambda_\B)$, real numbers $M, N > 0$ and an integer $1\le j\le n-1$, we introduce  the quantities
\begin{equation*}
\Sigma_1^j(M,N,\lambda) = \sumstar{ \A\sim M} \Big\vert \sumstar{ \B\sim N\ } \lambda_\B \chi_\A^j(\B) \Big\vert^2, \quad \quad
%
\Sigma_2^j(M,N,\lambda) = \sum_{\substack{ \A\sim M \\ \A \in \mathcal{I}(S)}} \Big\vert \sumstar{ \B\sim N} \lambda_\B \chi_\B^j(\A) \Big\vert^2.
\end{equation*}
For a class $\mathcal{C} \in R_{\mathfrak{c}}$ let
\begin{equation}\label{sigma3}
\Sigma_3^j(M,N,\lambda, \mathcal{C}) = 
\Bigl|
	\sum_{\substack{
		\B_1,\B_2\sim N\\ 
		\B_1,\B_2\in\I(S)\\
		(\B_1,\B_2)=(1)\\ 
		[\B_1] =  [\B_2] = \mathcal{C}
	}} 
	\lambda_{\B_1} \overline{\lambda_{\B_2}} 	
	\sum_{\A\subseteq \mathcal{O}} 
		W\left(\frac{\n \A}{M}\right) \chi_{\B_1,\B_2}^j(\A)
\Bigr|.
\end{equation}
We write 
\begin{equation*}
B_i^j(M,N) = \sup_{\lambda \not = 0} \Bigl\{\Sigma_i^j(M,N,\lambda) / \sumstar{\B\sim N } |\lambda_\B |^2\Bigr\}, \quad i = 1, 2,
\end{equation*}
and
\begin{equation*}
B_3^j(M,N) = \sup_{\lambda \not = 0} \max_{\mathcal{C} \in R_{\mathfrak{c}}} \Bigl\{\Sigma_3^j(M,N,\lambda, \mathcal{C}) / \sumstar{ \B\sim N, [\B] = \mathcal{C} } |\lambda_\B |^2\Bigr\}.
\end{equation*}
Our  objective in this subsection is to state properties of these norms and  some explicit relations between them.

\begin{lem}\label{basiclemma}
We have
\begin{equation}\label{duality}
  B_1^j(M, N) \ll B_1^j(N, M).
\end{equation}
Moreover, there exists a constant $C \geq 1$ such that whenever  $M_2 \ge C M_1 \log(2M_1N)$,
\begin{equation}\label{mono}
B_1^j(M_1,N) \ll B_1^j(M_2,N) .
\end{equation}
\end{lem}

This is \cite[Lemmas 2.4 and 2.5]{GL} which in turn is a slight generalization of \cite[Lemmas 4 and 5]{HB1}.  From the reciprocity formula in Lemma \ref{lemma21} one concludes easily
\begin{equation}\label{B12}
    B_1^j(M,N) \ll B_2^j(M,N)
\end{equation}
after splitting the $\B$-sum into classes modulo $R_{\mathfrak{c}}$. Next we bound $B_2$ in terms of $B_3$:

\begin{lem}\label{lem:B2-B3}
There exist $1\le \Delta_1\ll \Delta_2$ such that
\begin{equation*}
B_2^j(M,N) \ll N^\varepsilon B_3^j\left(\frac{M}{\Delta_1},\frac{N}{\Delta_2}\right).
\end{equation*}
\end{lem}

\begin{proof} Fix a sequence $\lambda_{\B}$ with $\B \in \mathcal{I}(S)$, $\B \sim N$ such that $\sum_{\B} |\lambda_{\B}|^2= 1$. By Cauchy's inequality we have
\begin{displaymath}
   \sum_{\substack{\A\sim M\\ \A\in \I(S)}} \Big\vert \sumstar{ \B\sim N\ } \lambda_\B \chi_\B^j(\A) \Big\vert^2 \ll \max_{\mathcal{C} \in R_{\mathfrak{c}}} \sum_{\substack{\A\sim M\\ \A\in \I(S)}} \Big\vert \sumstar{\substack{\B\sim N\\   [\B] = \mathcal{C}}} \lambda_\B \chi_\B^j(\A) \Big\vert^2.
\end{displaymath}
We closely follow the proof of \cite[Lemma 7]{HB1}, and present only the key steps. Insert the factor $W(\n\A/M)$, open the square and sort by $\mathfrak{d} = (\B_1, \B_2)$. Writing 
 $\mathcal{C}_{\mathfrak{d}} := \mathcal{C} [\mathfrak{d}]^{-1} \in R_{\frakc}$, we find that the previous display is bounded by
\begin{displaymath}
  \max_{\mathcal{C} \in R_{\mathfrak{c}}} \sum_{\mathfrak{d} \in \mathcal{I}(S)}  \sum_{\substack{  \A\in \I(S)\\ (\A, \mathfrak{d}) = (1)}} W\left(\frac{\n\A}{M}\right) \sumstar{\substack{\B_1, \B_2 \sim N/\n\mathfrak{d}\\ [\B_1]= [\B_2]= \mathcal{C}_{\mathfrak{d}}  \\  (\B_1, \B_2) = (\mathfrak{d}, \B_1\B_2) = (1)}} \lambda_{\B_1\mathfrak{d}}\overline{\lambda_{\B_2\mathfrak{d}}}  \chi_{\B_1, \B_2}^j(\A).
\end{displaymath}
Note that we could replace $\chi_{\B_1} \bar{\chi}_{\B_2}$ by its underlying primitive character $\chi_{\B_1, \B_2}$ since we are only summing over $\A \in \mathcal{I}(S)$.
Now we remove the coprimality conditions in the $\A$-sum by M\"obius inversion. Let $\mathfrak{s}$ denote the product of the primes contained in the set $S$. Recalling that in our situation $\chi_{\B_1, \B_2} = \chi^{(\B_1)} \bar{\chi}^{(\B_2)}$, we writing $\lambda^{\ast}_{\B} := \lambda_{\B\mathfrak{d}}(\chi^{(\B)})^j(\mathfrak{r})$ and recast  the triple $\A, \B_1, \B_2$-sum as
\begin{displaymath}
  \sum_{\mathfrak{r} \mid \mathfrak{s}\mathfrak{d}} \mu(\mathfrak{r})  \sum_{\A\subseteq \mathcal{O}} W\left(\frac{\n\A}{M/\n\mathfrak{r}}\right) \sumstar{\substack{\B_1, \B_2 \sim N/\n\mathfrak{d}\\  [\B_1] = [\B_2] = \mathcal{C}_{\mathfrak{d}}\\   (\B_1, \B_2) = (\mathfrak{d}, \B_1\B_2) = (1)}} \lambda^{\ast}_{\B_1}\overline{\lambda^{\ast}_{\B_2}}  \chi_{\B_1, \B_2}^j(\A).
\end{displaymath}
Hence
\begin{displaymath}
  \Sigma_2^j(M, N, \lambda) \ll \max_{1 \leq \Delta_1 \ll \Delta_2} B_3^j\left(\frac{M}{\Delta_1}, \frac{N}{\Delta_2}\right)  \sum_{\mathfrak{d} \in \mathcal{I}(S)}\sum_{\mathfrak{r} \mid \mathfrak{s}\mathfrak{d}} \sum_{\substack{\B \in \mathcal{I}(S)\\ \B \sim N/\n\mathfrak{d}}} |\lambda_{\B\mathfrak{d}}|^2,
\end{displaymath}
and the triple sum is $\ll N^{\varepsilon} \sum_{\B}|\lambda_{\B}|^2 = N^{\varepsilon}$.
\end{proof}

\begin{lem}\label{lem:B3-B2}
The following bound holds:
\begin{equation*}
B_3^j(M,N) \ll M+ (MN)^\varepsilon \frac{M}{N} \max_{1\le K\le (MN)^{\varepsilon}N^2/M} B_2^j(K,N).
\end{equation*}
\end{lem}

\begin{proof}
If $N\le1$, we estimate trivially, getting $B_3^j(M,N)\ll M$. Assume now $N>1$ and fix a sequence $\lambda_{\B}$ with $\B \sim N$, $[\B] =\mathcal{C}$ for some fixed class $\mathcal{C} \in R_{\mathfrak{c}}$, $\B \in \mathcal{I}(S)$ such that $\sum_{\B} |\lambda_{\B}|^2 = 1$.  These assumptions imply  that  the character $\chi^j_{\B_1, \B_2}$ 
in the definition of $\Sigma_3^j(M, N, \lambda, \mathcal{C})$ is a non-trivial primitive character with conductor $\B_1\B_2$ (cf.\ Lemma \ref{conductor}). Using Lemma \ref{poisson} and \eqref{gauss} for the inner sum in \eqref{sigma3}, we obtain
\begin{equation}\label{eq:B3-B2:4}
\begin{split}
\Sigma_3^j(M,N,\lambda, \mathcal{C}) \ll M  \sum_{\A\subseteq \mathcal{O}}& \Big\vert \sumstar{\substack{  \B_1, \B_2 \sim N\\(\B_1,\B_2)=(1)\\ [\B_1] =  [\B_2] = \mathcal{C}}}\epsilon((\chi^{(\B_1)})^j)\epsilon((\bar{\chi}^{(\B_2)})^j)\bar{\psi}_{(\B_2)}(\B_2)^j \psi_{(\B_1)}(\B_1)^j  \\
&\times  \lambda_{\B_1}\overline{\lambda_{\B_2}}\dot{W}\left(\frac{M\n\A}{\n(\B_1\B_2)}\right) \frac{\overline{\chi}^j_{\B_1,\B_2}(\A)}{\sqrt{\n(\B_1\B_2)}}\Big\vert.
\end{split}
\end{equation}
We factorize $\A = \A_0 \mathfrak{s}$ where $\A_0 \in \mathcal{I}(S)$ and $\mathfrak{s} \mid \mathfrak{c}^{\infty}$. Next we split the $\A_0$-sum into dyadic ranges $\A \sim K$ for $K \geq 1/2$.   Note that the term $\A = (0)$ does not occur  in \eqref{eq:B3-B2:4}, because $\bar{\chi}^j_{\B_1, \B_2}$ is non-trivial.  For notational simplicity let us write $Z := (MN)^{\varepsilon}$.  By \eqref{eq:B3-B2:2} we can neglect the values of $K$ with $K> ZN^2/(M\n\mathfrak{s})$, at the cost of an error $(MN)^{-A}$. We insert the definition of $\dot{W}$ and shift the contour to $\Re s = \varepsilon$. By the rapid decay of $\widehat{W}$ we can truncate the integral at $|\Im s| \leq Z^{\varepsilon}$, again at the cost of an error $(MN)^{-A}$. 
In this way we bound the right hand side of \eqref{eq:B3-B2:4} by
\begin{equation*}
(MN)^{-A} + Z M\sum_{\mathfrak{s} \mid \mathfrak{c}^{\infty}} \max_{K\le \frac{ZN^2}{M\n\mathfrak{s}}} \max_{|t| \leq Z} \sum_{\substack{\A_0\in\I(S)\\ \A_0\sim K}} \big\vert\sumstar{\substack{  \B_1, \B_2 \sim N\\ (\B_1,\B_2)=(1)\\ [\B_1]= [\B_2] = \mathcal{C}}} \lambda^{\ast}_{\B_1} \lambda^{\ast\ast}_{\B_2} \frac{\overline{\chi}^j_{\B_1}(\A_0)\chi^j_{\B_2}(\A_0)}{\n(\B_1\B_2)^{1/2+it-\varepsilon}}\Big\vert
\end{equation*}
where
\begin{displaymath}
  \lambda^{\ast}_{\B} := \lambda_{\B}  \epsilon((\chi^{(\B)})^j)\psi_{(\B)}(\B)^j \bar{\chi}^{(\B)}(\mathfrak{s})^j \quad \text{and} \quad \lambda^{\ast\ast }_{\B} := \overline{\lambda_{\B}}   \epsilon((\bar{\chi}^{(\B)})^j)\bar{\psi}_{(\B)}(\B)^j\chi^{(\B)}(\mathfrak{s})^j .
\end{displaymath}
   Note that we replaced $\bar{\chi}_{\B_1, \B_2}(\A_0)$ by $\bar{\chi}_{\B_1} \chi_{\B_2}(\A_0)$ since $\A_0 \in \mathcal{I}(S)$.
Detecting coprimality with the M\"obius function and using Cauchy's inequality, we conclude verbatim as in \cite[p.\ 121]{HB1} that the $\A, \B_1, \B_2$-sum is bounded by 
\begin{align*}
\ll \sum_{\n \mathfrak{d} \le 2N} B_2^j(K,N) \sumstar{  \mathfrak{d} \mid \B} |\lambda_\B|^2 N^{2\varepsilon-1} \ll N^{3\varepsilon-1} B_2^j(K, N).
\end{align*}
It remains to observe that $\# \{\mathfrak{s} \mid \mathfrak{c}^{\infty} : \n\mathfrak{s} \leq ZN^2\} \ll (ZN^2)^{\varepsilon}$.
\end{proof}

In order to state our last result, we need some more notation. Given $M\ge 1/2$, let
\begin{equation*}
\Xi = \left\{ \textbf{X}=(X_1,\dots,X_n)\in [1/2,\infty)^n
\,:\, 
\prod_{i=1}^n X_i^i<M \leq \frac{1}{2}   \prod_{i=1}^n (2X_i)^i
\right\}.
\end{equation*}
For $\textbf{X}=(X_1,\dots,X_n)\in \Xi$, we write $X=\prod_{i=1}^n X_i$, and we define $\ell =\ell(\textbf{X})$ by $\ell=1$ or $\ell=2$, according to if $X_1\ge X_2$ or $X_1<X_2$. In particular, $X_{\ell} = \max(X_1, X_2)$.

\begin{lem}\label{lem:B2-B1}
Suppose that $j \not= n/2$. Then
\begin{equation*}
B_2^j(M,N) \ll (MN)^\varepsilon \max_{\textbf{X} \in \Xi} \left\{ \min \left(M^{1/3},\Bigl(\frac{M}{X_\ell}\Bigr)^{2/3} \right) B_1^{j\ell}(X_\ell,N)\right\}.
\end{equation*}
\end{lem}

\begin{proof}
Recall that any ideal $\A\in I(S)$ can be decomposed uniquely as $\A=\A_1\A_2^2 \dots \A_{n-1}^{n-1} \A_n^n$, with $\A_i$ square-free for all $i=1,\dots, n-1$. 
For $\textbf{X}=(X_1,\dots,X_n) \in \Xi$ define
\begin{equation*}
\Sigma_4^j(M,N,\lambda,\textbf{X}) = 
\sumstar{\A_1 \sim X_1} \cdots \sumstar{\A_{n-1} \sim X_{n-1}} \sum_{\substack{\A_n \sim X_n\\ \A_n \in \mathcal{I}(S)}} \Big\vert \sumstar{\B\sim N} \lambda_\B \chi_\B(\A_1\A_2^2 \dots \A_{n-1}^{n-1} \A_n^n)\Big\vert^2.
\end{equation*}
Then clearly $\Sigma_2^j(M,N,\lambda) \ll \log(2+M)^{n-1}\max_{\textbf{X}\in \Xi} \Sigma_4^j(M,N,\lambda,\textbf{X})$. Let $1\le \ell\le 2$ be as defined above and consider only the $\A_\ell$-sum together with the $\B$-sum, the other variables being fixed. Define a new sequence $\lambda^{\ast}$ by $\lambda_\B^{\ast}= \lambda_\B \chi_\B(\A/\A_{\ell}^{\ell})$. 
Then one sees that
\begin{equation*}
\Sigma_4^j(M,N,\lambda,\textbf{X}) =
\sumstar{\substack{
	\A_j \sim X_j\\ 
	1 \leq j \leq n-1\\ 
	j \not=\ell
}} 
\sum_{\substack{\A_n \sim X_n\\ \A_n \in \mathcal{I}(S)}} \Sigma_1^{j\ell}(X_\ell,N,\lambda^{\ast}).
\end{equation*}
It is here that we need $j \neq n/2$, in order to ensure that $j\ell \not\equiv 0$ (mod $n$). Since $|\lambda_\B^{\ast}| \leq |\lambda_{\B}|$, it follows that
\begin{equation}\label{eq:B2-B1:1}
B_2^j(M,N) \ll (MN)^\varepsilon \max_{\textbf{X}\in\Xi} \ \frac{X}{X_\ell} B_1^{j\ell}(X_\ell,N).
\end{equation}
Recall that by definition, $X_1X_2^2\dots X_n^n\leq M$ for all $\textbf{X} \in \Xi$. If 
$\ell=1$, i.e. $X_1\ge X_2$, one has
\begin{equation}\label{eq:B2-B1:2}
\frac{X}{X_1}= (X_2^2\dots X_n^2)^{1/2} \leq \left(\frac{M}{X_1}\right)^{1/2}\ll \left(\frac{M}{X_1}\right)^{2/3}
\end{equation}
as well as
\begin{equation}\label{eq:B2-B1:3}
\frac{X}{X_1}= (X_2^3\dots X_n^3)^{1/3} \leq \left(\frac{MX_2}{X_1}\right)^{1/3}\leq M^{1/3}.
\end{equation}
If $\ell=2$, i.e. $X_1<X_2$, one has
\begin{equation}\label{eq:B2-B1:4}
\frac{X}{X_2}= (X_1^3X_3^3\dots X_n^3)^{1/3} \leq \left(\frac{MX_1^2}{X_2^2}\right)^{1/3}\leq M^{1/3}\ll \left(\frac{M}{X_2}\right)^{2/3}.
\end{equation}
The lemma now follows from \eqref{eq:B2-B1:1} --  \eqref{eq:B2-B1:4}.
\end{proof}

\subsection{The recursive argument}
\label{subsect:rec}

The idea of the proof of Theorem \ref{thm:thm3} is to improve iteratively on the bounds
\begin{equation*}\tag{$E_\alpha$}
B_1^j(M,N) \ll (MN)^\varepsilon (M^\alpha+N+(MN)^{2/3}) 
\end{equation*}
and
\begin{equation*}\tag{$F_\alpha$}
B_1^j(M,N) \ll (MN)^\varepsilon (M+N^\alpha+(MN)^{2/3}). 
\end{equation*}
By \eqref{duality}, the statements $(E_{\alpha})$ and $(F_{\alpha})$ are equivalent.


\begin{lem}\label{rec}
Assume that $(E_\alpha)$ holds for some $\alpha > 4/3$. Then $(E_{2 - \frac{2}{3\alpha-1}})$ holds.
\end{lem}

\begin{proof}
Let us assume $(E_{\alpha})$. As a first step we will show that
\begin{equation}\label{eq:rec1}
B_1^j(M,N) \ll B_2^j(M,N)   \ll  (MN)^\varepsilon \left(M+ N^{2\alpha-1}M^{1-\alpha} +(MN)^{2/3}\right).
\end{equation}
The first inequality follows from \eqref{B12}. In order to verify the second, let us first assume $j \not= n/2$. Then
Lemma \ref{lem:B2-B1} implies
\begin{align*}
B_2^j(M,N)
&\ll (MN)^\varepsilon \max_{\textbf{X}\in\Xi} \ \min \left(M^{1/3},\Bigl(\frac{M}{X_\ell}\Bigr)^{2/3} \right) B_1^{j\ell}(X_\ell,N)\\
&\ll (MN)^\varepsilon \max_{\textbf{X}\in\Xi} \ \min \left(M^{1/3},\Bigl(\frac{M}{X_\ell}\Bigr)^{2/3} \right) (X_\ell^\alpha+N+ (X_\ell N)^{2/3})\\
&\ll (MN)^\varepsilon \max_{\textbf{X}\in\Xi} \ \left(\Bigl(\frac{M}{X_\ell}\Bigr)^\alpha X_\ell^\alpha+ M^{1/3}N + \Bigl(\frac{M}{X_\ell}\Bigr)^{2/3} (X_\ell N)^{2/3}\right)\\
&= (MN)^\varepsilon \left(M^\alpha +M^{1/3}N + (MN)^{2/3}\right)\ll (MN)^\varepsilon \left(M^\alpha +M^{1/3}N\right).
\end{align*}
Substituting in Lemma \ref{lem:B3-B2} gives
\begin{equation*}
B_3^j(M,N) \ll  (MN)^\varepsilon \left(M+ N^{2\alpha-1}M^{1-\alpha} +(MN)^{2/3}\right).
\end{equation*}
Finally, we apply Lemma \ref{lem:B2-B3}. By noting that $\Delta_1^{\alpha-1}\Delta_2^{1-2\alpha}= (\Delta_2/\Delta_1)^{1-\alpha} \Delta_2^{-\alpha} \ll 1$, we arrive at \eqref{eq:rec1} under the extra assumption $j \not= n/2$.

If $j=n/2$, then we have from \cite[Corollary 1.5]{GL} (which is applicable because of Lemma \ref{lemma21} and \ref{conductor}) the optimal estimate $B_1^{j} (M,N) \ll (MN)^\varepsilon(M+N)$. 
Therefore, \eqref{eq:rec1} holds in all cases.\\

We can now distinguish two cases in \eqref{eq:rec1}. If $M\ge  N^{2-3/(3\alpha-1)}$, then the second term on the right hand side of \eqref{eq:rec1} is bounded by the third, and $(E_{2-\frac{2}{3\alpha-1}})$ follows a fortiori. On the other hand, if $M\le  N^{2-3/(3\alpha-1)} $, we use \eqref{mono} to bound the left hand side of \eqref{eq:rec1} by
\begin{displaymath}
   B_1^j(C N^{2-3/(3\alpha-1)}  \log(2NM) ,N),
\end{displaymath}
which by \eqref{eq:rec1} is in turn bounded by $(MN)^\varepsilon (M+N^{2-2/(3\alpha-1)}+(MN)^{2/3})$. Hence $(F_{2-\frac{2}{3\alpha-1}})$ is true and the lemma follows by duality.
\end{proof}

To complete the proof of Theorem \ref{thm:thm3}, we recall from \cite[Lemma 2.2]{GL} (see also \cite[(21)]{HB1}) that $(E_2)$ holds. Thus, by iterating Lemma \ref{rec}, we deduce that $(E_\alpha)$ holds for the fixed point of $\alpha \mapsto 2-2/(3\alpha-1)$, i.e. for $\alpha=4/3$. This completes the proof of Theorem \ref{thm:thm3}.

\section{The double Dirichlet series of $n$-th order characters}
\label{sect:DoubleDirichlet}

The goal of this section is to prove Theorems \ref{thm:sub} and \ref{thm:sub1}. 
We begin by reviewing the work of Friedberg, Hoffstein, and Lieman \cite{FHL}: we give precise definitions of the double Dirichlet series $Z_1$ and $Z_2$ and their completed series $\Xi_1$ and $\Xi_2$, and state the functional equations they satisfy. Following this, we prove approximate functional equations for these series. From here it becomes a relatively simple matter to deduce our main theorems. 

\subsection{The functional equations}


Recall that we defined
$Z_1(s, w; \psi, \psi')$ in \eqref{eq:defz1}, but did not give a precise definition of its companion series $Z_2$; we do so now. Let $\epsilon(\chi_{\A})$ be the epsilon-factor (at the point $s=1/2$) in the functional equation of $L(s, \chi_{\A})$, and define
\[
    g(\A, \B) = \prod_{\substack{\pr^{\alpha} \parallel \A  \\  \pr^{\beta} \parallel \B}} g(\pr^{\alpha}, \pr^{\beta})
\]
where
\begin{equation}
\label{eq:gauss}
    g(\pr^{\alpha}, \pr^{\beta})  =
    \begin{cases}
        1, & \beta = 0, \\
        \n\pr^{\beta/2 - 1} (\n\pr - 1), & \alpha \geq \beta, \, \beta \equiv 0\mod{n}, \,
            \beta > 0,\\
        - \n\pr^{\beta/2 - 1}  , & \alpha = \beta-1, \, \beta \equiv 0\mod{n}, \, \beta > 0,\\
        \n\pr^{(\beta-1)/2 }  , & \alpha = \beta-1, \, \beta \not\equiv 0\mod{n}, \, \beta > 0,\\
        0, & \text{otherwise}.
\end{cases}
\end{equation}
With this notation let
\[
    D(w, \A, \psi) = \zeta_{K, S}(nw-n/2+1) \sum_{\B \in \I(S)}
    \frac{\epsilon(\chi_\B) \overline{\chi_{\B}}(\A^*) \psi(\B)g(\A, \B)}{\n\B^w}
\]
with $\A$, $\psi$ as above, $\A^*$ the part of $\A$ coprime to the $n$-th power free part of $\B$, $\zeta_{K, S}$ the Dedekind zeta-function attached to $K$ without Euler factor at primes in $S$, and $\Re w$ sufficiently large. Then we define
\begin{equation}
\label{eq:defz2}
    Z_2(s, w; \psi, \psi') = \sum_{\A \in \I(S)} \frac{D(w, \A, \psi') \psi(\A)}{\n\A^s}.
\end{equation}

We next define the completed series $\Xi_j$ of the double Dirichlet series $Z_j$:
\[
    \Xi_j(s, w; \psi, \psi') := G_j(s, w) Z_j(s, w; \psi, \psi')
\]
where the $G_j$ are local factors at infinity defined by:
\[
    G_1(s, w) :=
        \left(
            \frac{\Gamma(s) \Gamma(w) \Gamma\big(n(s+w-1)\big)}
                {\Gamma(s+w-1) (2\pi)^{3/2} (2 \pi n)^{n(s+w-1)-1/2}}
        \right)^{d/2}
\]
and
\[
    G_2(s, w) :=
        \left(
            \frac{\Gamma(s) \Gamma(s+w-1/2) \Gamma\big(n(w-1/2)\big)}
                {\Gamma(w-1/2) (2\pi)^{2s+1} (2 \pi n)^{n(w-1/2) - 1/2}}
        \right)^{d/2}.
\]
For future reference we note that
\begin{equation}
\label{eq:quotient}
    \frac{G_2(1-s, w+s-1/2)}{G_1(s, w)} =
        \left(\frac{\Gamma(1-s)}{\Gamma(s)(2 \pi)^{3/2-2s}}\right)^{d/2} .
\end{equation}
Having given precise definitions, we now summarize the analytic properties of these series (see \cite[Theorem 1.4 and Section 2]{FHL}).
Let
\[
  \zeta_K^{S}(s) := \prod_{\pr \in S} \left(1- \frac{1}{\n\pr^s}\right)^{-1}.
\]

\begin{prop}
Let $\psi, \psi'$ be two characters of $R_{\frakc}$. Then the function $Z_1(s, w; \psi, \psi')$ is holomorphic on $\C^2$, and of moderate growth in fixed vertical tubes except for possible (simple) polar divisors at $s=1$, $w=1$, and $s+w = 1+1/n$. Similarly, the function $Z_2(s, w; \psi, \psi')$ is holomorphic on $\C^2$, and of moderate growth in fixed vertical tubes except for possible (simple) polar divisors at $s=1$, $w=1/2+1/n$, and $s+w = 3/2$. Moreover, for $\rho, \rho' \in \widehat{R_{\frakc}}$ there are complex numbers $\alpha_{\psi, \psi', \rho, \rho'}$  and holomorphic functions $\alpha^{2, 1}_{\psi, \rho}(s)$, $\alpha^{1, 2}_{\psi, \rho}(s)$, $\alpha^{2, 2}_{\psi', \rho}(s)$, bounded in fixed vertical strips, such that
\begin{equation}
\label{eq:fe11}
    Z_1(s, w; \psi, \psi') =
        \sum_{\rho, \rho' \in \widehat{R_{\frakc}}}
            \alpha_{\psi, \psi', \rho, \rho'} Z_1(w, s; \rho, \rho');
\end{equation}
\[
\begin{split}
    & \Xi_2\left(1-s, w+s-\frac{1}{2}; \psi, \psi'\right)
    =
    \zeta_K^{S}(ns)
        \sum_{\rho\in \widehat{R_{\frakc}}} \alpha^{2, 1}_{\psi, \rho}(s)
        \Xi_1(s, w; \overline{\psi}, \psi\psi'\rho);
\end{split}
\]
\begin{equation}
\label{eq:fe12}
\begin{split}
    & \Xi_1\left(1-s, w+s-\frac{1}{2}; \psi, \psi'\right)
    =
    \zeta_K^{S}(ns)
        \sum_{\rho\in \widehat{R_{\frakc}}} \alpha^{1, 2}_{\psi, \rho}(s)
        \Xi_2(s, w; \overline{\psi}, \psi\psi'\rho);
\end{split}
\end{equation}
\[
\begin{split}
    & \Xi_2\left(s+w-\frac{1}{2}, 1-w; \psi, \psi'\right)
    =
    \zeta_K^{S}\left(nw+1-\frac{n}{2}\right)
        \sum_{\rho\in \widehat{R_{\frakc}}} \alpha^{2, 2}_{\psi', \rho}(s)
        \Xi_2(s, w; \psi\psi'\rho,  \overline{\psi'}) .
    \end{split}
\]
Iterating these functional equations, one can find holomorphic functions
$\beta^{(1)}_{\psi, \psi', \rho, \rho'}(s, w)$
and
$\beta^{(2)}_{\psi, \psi', \rho, \rho'}(s, w)$,
bounded in vertical tubes, satisfying
\begin{equation}
\label{eq:fe1}
\begin{split}
    \Xi_1(s, w; \psi, \psi')
    =
    \zeta_K^{S}(n(1-s)) & \zeta_K^{S}(n(1-w)) \zeta_K^{S}(n(1+1/n-s-w)) \\
    & \times
        \sum_{\rho, \rho' \in \widehat{R_{\frakc}}} \beta^{(1)}_{\psi, \psi', \rho, \rho'}(s, w)
            \Xi_1(1-s, 1-w; \rho, \rho').
\end{split}
\end{equation}
\end{prop}

\subsection{Approximate functional equations}

In this section we apply the functional equations to obtain explicit expressions for the quantity $Z_1(s, w;\psi, \psi')$ on the plane $\Re s = \Re w = 1/2$. For $u, t \in \R$ let
\begin{equation}
\label{eq:defC}
    C_1 := \big((1+|u|)(1+|u+t|)^{n-1}\big)^{d}.
\end{equation}

\begin{lem}
Let $u, t \in \R$. There exist smooth functions $V^{\pm}: [0, \infty) \rightarrow \C$ depending on $u$ and $t$ satisfying
\begin{equation}
\label{eq:boundV}
    y^j \frac{d^j}{dy^j} V^{\pm}(y) \ll_{j, A} (1+y)^{-A}
\end{equation}
for all $j, A \in \N_0$ uniformly in $u$ and $t$, such that
\[
\begin{split}
    Z_1(1/2+it, 1/2+iu; \psi, \psi')
    \ll
    & \sum_{\rho, \rho' \in \widehat{R_{\frakc}}}
        \sum_{\pm}
            \Bigl|
                \sum_{\A} \frac{L^*(1/2 \pm it, \rho, \A) \rho'(\A)}{\n\A^{1/2\pm iu}} V^{\pm}\left(\frac{\n\A}{\sqrt{C_1}}\right)
            \Bigr|.
\end{split}
\]
\end{lem}

\noindent\textbf{Proof.}
For fixed $z \in \C$ with $\Re z \neq 0$ let
\[
    P_z(w) := \left(1-  2^{z-w}\right)\left(1- 2^{z+w}\right) \left(1 -2^z\right)^{-2}.
\]
This is an even holomorphic function, bounded in fixed vertical strips, with a zero at $w=z$ and  $P_z(0) = 1$. Let
$\mathfrak{z} := \left\{\pm 1/2- iu,  \pm 1/n - i(u+t) \right\}$.
For a large positive integer $A$ we define
\[
    H_{t, u}(w) :=
        \left(\cos\frac{\pi w}{3A}\right)^{-3And}
            \prod_{z \in \mathfrak{z}} P_z(w).
\]
This is an even function, satisfying $H_{t, u}(0) = 1$ with simple zeros at all $w \in \mathfrak{z}$. Moreover, $H_{t, u}$ is holomorphic in $| \Re w | < 3A/2$. After these preparations we write
\begin{multline}
\label{eq:residue}
    \Xi_1(1/2 + it, 1/2 + iu; \psi, \psi')
    = \frac{1}{2\pi i} \int_{(1)}
            \Xi_1(1/2 + it, 1/2 + iu+w; \psi, \psi') H_{t, u}(w)  \frac{dw}{w} \\
    - \frac{1}{2\pi i} \int_{(-1)}
            \Xi_1(1/2 + it, 1/2 + iu+w; \psi, \psi') H_{t, u}(w)  \frac{dw}{w}.
\end{multline}
The rapid decay of $H_{t, u}$ ensures the absolute convergence of the two integrals and that the integrand is holomorphic except at $w=0$, see \cite[Theorem 1.4]{FHL}.
Moreover, the Dirichlet series $Z_1(1/2 + \pm it, w; \psi, \psi')$ is absolutely convergent in
$\Re w > 1$. For the last term in \eqref{eq:residue} we apply the functional equation \eqref{eq:fe1}, change variables $w \mapsto -w$ and open both Dirichlet series. This gives the lemma with
\[
    V^{\pm}(y) =
        \frac{1}{2\pi i} \int_{(1)}
            \frac{G_1(1/2 \pm it, 1/2 \pm iu + w)}{G_1(1/2 + it, 1/2 + iu)}  H_{t, u}(w)  (\sqrt{C_1}y)^{-w} \frac{dw}{w}.
\]
It remains to establish \eqref{eq:boundV}. It is an easy consequence of Stirling's formula
(see e.g.\ \cite[p.\ 100]{IK}) that
\[
    \frac{\Gamma(s+z)}{\Gamma(s)} \ll_{\Re s, \Re z}
        (1 +|s|)^{\Re z} \exp\left(\frac{\pi}{2} |z|\right), \quad \Re s, \Re (s+z) \geq 1/100.
\]
We conclude that
\[
    \frac{G_1(1/2 \pm it, 1/2 \pm iu + w)}{G_1(1/2 + it, 1/2 + iu)} \ll
        C_1^{\Re w/2} \exp\left(\frac{\pi n d}{4}|w|\right)
\]
with $C_1$ as in \eqref{eq:defC}. Shifting the contour and differentiating under the integral sign we find
\[
    y^j\frac{d^j}{dy^j} V^{\pm}(y) \ll_{\sigma, j}
        y^{-\sigma} \int_{(\sigma)}
            \exp\left(\pi \Bigl(\frac{nd}{4} - nd\right) |w|\Bigr)(1+|w|)^{ j} \, \frac{dw}{w} + \delta_{j = 0, \sigma < 0}
\]
for any $-1/4 \leq \sigma < 3A/2$, $\sigma \not= 0$. The second term on the right hand side comes from the pole at $w = 0$ if $j = 0$. The $w$-integral is rapidly converging. For $y \leq 1$ we choose $\sigma = -1/8$, say, for $y > 1$ we choose $\sigma = A$ and arrive at \eqref{eq:boundV}.
\qed \\

Applying a smooth partition of unity, we conclude
\begin{equation}
\label{eq:max}
    Z_1(1/2+it, 1/2+iu; \psi, \psi') \ll
        C_1^{\varepsilon}
        \max_{\rho, \rho' \in \widehat{R_{\frakc}}}
        \left(
            \max_{1 \leq P \leq C_1^{1/2 +\varepsilon}}
                |D^{\pm}_{\rho, \rho'}(t, u, P)|
        \right)
\end{equation}
where
\[
    D^{\pm}_{\psi, \psi'}(t, u, P) :=
        \sum_{\A} \frac{L^*(1/2 \pm it, \psi, \A)  \psi'(\A)}{\n\A^{1/2\pm iu}}
            W\Big(\frac{\n\A}{P}\Big)
\]
for a certain function $W$ with support in $[1, 2]$ depending on $t, u, P$, all of whose derivatives are  bounded uniformly in $t, u, P$. By Mellin inversion we have
\[
    D^{\pm}_{\psi, \psi'}(t, u, P) =
        \frac{1}{2\pi i} \int_{(3)}
            Z_1(1/2 \pm it, 1/2 \pm iu + w, \psi, \psi') \widehat{W}(w) P^w  dw
\]
where $\widehat{W}$ is holomorphic and rapidly decaying in vertical strips. Fix an even holomorphic function $H_t$ with $H_t(0) = 1$ and $H_t(1/2 \pm it) = 0$ which decays rapidly in vertical strips, and let $R>0$ be a parameter.  By Cauchy's theorem we have
\begin{equation}
\label{eq:decompD}
    D^{\pm}_{\psi, \psi'}(t, u, P)  = D_1 + D_2
\end{equation}
where
\[
\begin{split}
    D_1  &=
        \frac{1}{(2\pi i)^2} \int_{(1)} \int_{(3)} Z_1(1/2 \pm it + s, 1/2 \pm iu + w, \psi, \psi') \widehat{W}(w) P^w \frac{H_t(s)R^s}{s} dw \,ds, \\
    D_2 & =
        \frac{1}{(2\pi i)^2} \int_{(1)} \int_{(3)} Z_1(1/2 \pm it - s, 1/2 \pm iu + w, \psi, \psi') \widehat{W}(w) P^w \frac{H_t(s)R^{-s}}{s} dw \,ds.
\end{split}
\]
Note that for $\Re w = 3$ and $-1 \leq \Re s \leq 1$ the function
$Z_1(1/2 \pm it + s, 1/2 \pm iu + w, \psi, \psi')H_t(s)$
is holomorphic. We now choose
\begin{equation}
\label{eq:defR}
    R := (1+|t|)^{d / 2} \sqrt{P} .
\end{equation}

The following two lemmas serve as templates of how to bring sums of central values of $L$-functions into a form where Theorem \ref{thm:thm3} can be applied. We recall the definition \eqref{eq:defC0} of $C$. We will use frequently Lemma \ref{lemma21}.

\begin{lem}
\label{lem:ref1}
    For any $\varepsilon > 0$ we have
    $D_1 \ll C^{\varepsilon}\big(P^{1/2} + R^{1/2} + (PR)^{1/3}\big)$.
\end{lem}

\noindent \textbf{Proof.}
By \eqref{eq:defz1} and \eqref{eq:defL*} we have
\[
    D_1 =
        \sum_{\substack{\A, \B \in \I(S)  }}
            \frac{\psi'(\A) \psi(\B) \chi_{\A}(\B)}{\n\B^{1/2 \pm it} \n\A^{1/2\pm iu}} W\Big(\frac{\n\A}{P}\Big) V\Big(\frac{\n\B}{R}\Big)
\]
where
\[
    V(x)
    =
    V^{\pm}_{t, \psi, \A}(x)
    =
    \frac{1}{2\pi i} \int_{(1)} H_t(s) \frak{A}(1/2 \pm it + s, \psi, \A) x^{-s} \frac{ds}{s}
\]
and $\frak{A}$ is defined in \eqref{eq:defa}.
By the rapid decay of $V$ and $W$ we can truncate the double sum at
$\n\A \leq P' := P C^{\varepsilon}$ and $\n\B \leq R' := RC^{\varepsilon}$ at the cost of a negligible error. Having done so, we shift the contour of the integral defining $V$ to the line
$\Re s = \varepsilon$ and truncate it at $|\Im s| \leq C^{\varepsilon}$ at the cost of a negligible error because of the rapid decay of $H_t$. Similarly we shift the contour of the integral defining $W$ to the line $\Re w = 0$ and truncate at $|\Im w| \leq C^{\varepsilon}$.  Hence
\[
    D_1 \ll
        C^{3\varepsilon}  \sup_{|y_1|, |y_2| \leq C^{\varepsilon}}
        \Bigl|
            \sum_{\substack{\n\A \leq P', \n\B \leq R'  \\  \A, \B \in \I(S)}}
                \frac{ \psi'(\A)\frak{A}(1/2 + \varepsilon \pm it + iy_2, \psi, \A) \psi(\B)    \chi_{\A}(\B)}{\n\A^{1/2\pm it + iy_1 } \n\B^{1/2+\varepsilon \pm iu +iy_2}}
        \Bigr|
        + C^{-A}.
\]
We decompose uniquely $\A = \A_0\A_1\A_2^n$ where $\A_0 \A_1$ is $n$-th power free, $\A_0$ is squarefree, $\A_1$ is squarefull and $(\A_0, \A_1) = 1$.  Similarly we decompose $\B = \B_0 \B_1$ where $\B_0$ is squarefree, $\B_1$ is squarefull and $(\B_0, \B_1) = 1$. Then the sum over $\A$ and $\B$ is bounded by (cf.\ \eqref{eq:defa})
\begin{equation}
\label{eq:sum}
    \sum_{ \n\A_1\A_2^n \leq P'} \frac{\n\A_2^{n/2 - 1 + \varepsilon}}{(\n\A_1\A_2^n)^{1/2}}
        \sum_{\n\B_1 \leq R'} \frac{1}{\n\B_1^{1/2}}
        \Bigl|
            \sum_{\substack{\n\A_0 \leq \frac{P'}{\n\A_1\A_2^n}  \\  (\A_0, \A_1) = 1}} \sum_{\substack{\n\B_0 \leq \frac{R'}{\n\B_1}  \\  (\B_0, \B_1) = 1}}
                \frac{ \psi'(\A_0) \psi(\B_0) \chi_{\A_0\A_1}(\B_0\B_1)}{\n\A_0^{1/2\pm it + iy_1 } \n\B_0^{1/2+\varepsilon \pm iu +iy_2}}
        \Bigr|.
\end{equation}
Here it is understood that $\B_0, \A_0$ are squarefree, $\B_1, \A_1$ are squarefull, $\A_1$ is $n$-th power free, and all ideals are coprime to $S$. We split the inner double sum over $\A_0$ and $\B_0$ into finitely many subsums according to their image in $R_{\frakc}$. By the reciprocity law stated in Lemma \ref{lemma21} we see that this double sum is bounded by
\[
    \ll\max_{\mathcal{C}_1, \mathcal{C}_2 \in R_{\frakc}}
        \Bigl|
            \sum_{\substack{\n\A_0 \leq \frac{P'}{\n\A_1\A_2^n}  \\  (\A_0, \A_1) = 1, [\A_0 ]= \mathcal{C}_1}}
            \sum_{\substack{\n\B_0 \leq \frac{R'}{\n\B_1}  \\  (\B_0, \B_1) = 1, [\B_0] = \mathcal{C}_2}}
                \frac{\chi_{\A_0}(\B_0)\chi_{\A_1}(\B_0)\chi_{\A_0}(\B_1)}{\n\A_0^{1/2\pm it + iy_1 } \n\B_0^{1/2+\varepsilon \pm iu +iy_2}}
        \Bigr|,
\]
and by Cauchy-Schwarz the double sum is at most
\[
    \ll (\log P')
    \Bigl(
        \sum_{\A_0}
        \Bigl|
            \sum_{\B_0} \frac{\chi_{\A_0}(\B_0) \chi_{\A_1}(\B_0) }{\n\B_0^{1/2+\varepsilon \pm iu +iy_2}}
        \Bigr|^2
    \Bigr)^{1/2}
\]
with the same summation conditions as in the previous display.  Applying Theorem \ref{thm:thm3} and substituting back into \eqref{eq:sum}, we obtain the lemma.
\qed \\

\begin{lem}
\label{lem:ref2}
    For any $\varepsilon > 0$ we have
    $D_2 \ll  C^{\varepsilon} \big(P^{1/2} + R^{1/2} + (PR)^{1/3}\big)$.
\end{lem}

\noindent \textbf{Proof.}
By \eqref{eq:fe12} and \eqref{eq:quotient} we have
\begin{multline*}
    D_2 =  \sum_{\rho \in \widehat{R_{\frakc}}} \frac{1}{(2\pi i)^2} \int_{(1)} \int_{(3)}
    c_{t, u}(s, w)  \widehat{W}(w) (1+|t|)^{s d} P^w \frac{H_t(s)R^{-s}}{s} \\
    \times
    Z_2(1/2 \mp it + s, 1/2 \pm i (u+t) + w-s; \overline{\psi}, \psi\psi' \rho) \, dw \, ds .
\end{multline*}
where
\[
\begin{split}
    c_{t, u}(s, w)
    & =
    \alpha_{\psi, \rho}^{1, 2}(1/2 \mp it + s)\zeta_K^{S}(n(1/2 \mp it + s))
    \left(
        \frac{\Gamma(1/2 \mp it + s)}{\Gamma(1/2\pm it - s)
        (2\pi)^{1/2 \mp 2 i t +2s}(1+|t|)^{2s}}
    \right)^{d/2} \\
    & \ll_{\Re s,\, \Re w} ( 1+ |\Im s|)^{d\, \Re s}
\end{split}
\]
in $\Re s , \Re w > 0$, uniformly in $u, t$. Now we insert the definition of $R$ given in \eqref{eq:defR} and change variables, getting
\begin{multline*}
    D_2 =
    \sum_{\rho \in \widehat{R_{\frakc}}} \frac{1}{(2\pi i)^2} \int_{(1)} \int_{(2)}
    c_{t, u}(s, w+s)  \widehat{W}(w+s)  P^w \frac{H_t(s)R^{s}}{s} \\
        \times Z_2(1/2 \mp it + s, 1/2 \pm i (u+t) + w; \overline{\psi}, \psi\psi' \rho) \, dw \, ds.
\end{multline*}
We open the absolutely convergent Dirichlet series and write
\[
    D_2 =
    \sum_{\rho \in \widehat{R}_{\frakc}} \sum_{\A, \B \in \I(S)}
        \frac{\epsilon(\chi_\B) \bar{\chi}_{\B}(\A^{\ast}) \psi\psi'\rho(\B)
        g(\A, \B)\bar{\psi}(\A)}{\n\A^{1/2 \mp it}\n\B^{1/2 \pm i(u+t)}}
        \Omega\Big(\frac{\n\A}{R}, \frac{\n\B}{P}\Big)
\]
where
\[
    \Omega(x, y) =
    \frac{1}{(2 \pi i)^2} \int_{(1)} \int_{(2)}c_{t, u}(s, w+s) \zeta_{K, S} (1  \pm in(u+t) + n s)   \widehat{W}(w+s)  H_t(s) x^{-s} y^{-w} \frac{ dw \, ds}{s}.
\]
With future manipulations in mind, we decompose uniquely $\B = \B'(\B'')^n$   with $(\B', \B'') = 1$ and
\begin{displaymath}
  \B'' = \prod_{\substack{\mathfrak{p}^{\beta} \parallel \B\\ \beta \equiv 0 \, (n)}} \mathfrak{p}^{\beta/n}.
\end{displaymath}
Accordingly we write $\A = \tilde{\A}\A'\A''$ where $(\tilde{\A}, \B) = 1$, $\A' \mid (\B')^{\infty}$ and $\A'' \mid (\B'')^{\infty}$. Shifting contours we see that $\Omega$ is rapidly decaying in both variables, so we can truncate\footnote{We could have truncated more canonically at $\n\A \leq R'$, but the above truncation is more convenient later. Note that $\tilde{\A}$ depends on $\B$, but this is not a problem.} $\n\tilde{\A} \leq R' = RC^{\varepsilon}$, $\n\B\leq P' = PC^{\varepsilon}$
at the cost of a negligible error. Now we shift the contour to $\Re w = 0$, $\Re s = \varepsilon$ and truncate the two contours at $|\Im w|, |\Im s| \leq C^{\varepsilon}$ at the cost of a negligible error. As above we conclude
\[
    D_2 \ll
        C^{3\varepsilon} \sup_{|y_1|, |y_2| \leq C^{\varepsilon}}
        \max_{\rho \in \widehat{R}_{\frakc}}
        \Bigl|
            \sum_{\substack{\B \in \I(S)  \\  \n\B \leq P'}}
            \sum_{\substack{\A  \in \I(S)  \\  \n\tilde{\A} \leq R' }}
            \frac{\epsilon(\chi_\B) \bar{\chi}_{\B}(\A^*) \psi\psi'\rho(\B)
            g(\A, \B)\bar{\psi}(\A)}{\n\A^{1/2+\varepsilon \mp it+i y_1}\n\B^{1/2 \pm i(u+t)+iy_2}}
        \Bigr|
        + C^{-A}.
\]
Our notation (cf.\ Lemma \ref{lemma21}) implies  $\chi_{\B}(\A^*) = \chi_{\B'}(\A^{\ast})= \chi_{\B'}(\tilde{\A}\A'')$ and
$g(\A, \B) = g(\A', \B')g(\A'', (\B'')^n)$. Let
\[
    G(\B') =
    \frac{\epsilon(\chi_\B') \chi_{\B'}(\A'') \psi\psi'\rho(\B')}{(\n\B')^{ \pm i(u+t)+iy_2} }
    \sum_{\A' \mid (\B')^{\infty}}
        \frac{g(\A', \B')\bar{\psi}(\A')}{(\n\A')^{1/2+\varepsilon \mp it+i y_1}}  \ll 1,
\]
cf.\ \eqref{eq:gauss}. Then we can bound the double sum by
\[
    \sum_{\n\B'' \leq R'} \frac{1}{(\n\B'')^{n/2}}
        \sum_{\A'' \mid (\B'')^{\infty}} \frac{g(\A'' , (\B'')^n )}{(\n\A'' )^{1/2}}
        \Bigl|
            \sum_{\substack{\n\B' \leq R'(\n\B'')^{-n}  \\  (\B', \B'') = 1}}
            \frac{G(\B')}{ (\n\B')^{1/2 } }
            \sum_{\substack{\n\tilde{\A} \leq P'  \\  (\tilde{\A}, \B'\B'') = 1}} \frac{\bar{\psi}(\tilde{\A}) \chi_{\B'}(\tilde{\A})}{(\n\tilde{\A})^{1/2+\varepsilon \mp it+i y_1}}
        \Bigr|.
\]
Here all ideals are understood to be coprime to $S$. We now proceed exactly as in the previous lemma: we decompose the inner sum into $\B' = \B'_0\B'_1(\B'_2)^n$ and $\tilde{\A} = \tilde{\A}_0 \tilde{\A}_1$ where $\B'_0, \tilde{\A}_0$ are squarefree, $\B'_1, \tilde{\A}_1$ are squarefull, $\B_0'\B_1'$ is $n$-th power free and  $(\B'_0, \B'_1) = (\tilde{\A}_0, \tilde{\A}_1) = 1$.  We pull the $\B'_1, \B_2', \tilde{\A}_1$-sums outside the absolute values, split the inner sums into finitely many subsums according to image of $\B'_0, \tilde{\A}_0$ in $R_{\frakc}$ and use this to decompose $\chi_{\B'}(\tilde{\A})$. Then we use Theorem \ref{thm:thm3} for the innermost double sum and estimate the rest trivially.
\qed

\subsection{Proofs of Theorems \ref{thm:sub} and \ref{thm:sub1}}

Combining \eqref{eq:defC0}, \eqref{eq:defC}, \eqref{eq:max}, \eqref{eq:decompD}, \eqref{eq:defR}, and Lemmas \ref{lem:ref1} and \ref{lem:ref2} we find
\begin{multline*}
    Z_1(1/2 + it, 1/2 + iu; \psi, \psi')
    \ll
    \max_{P \ll C_1^{1/2+\varepsilon}}
    \left(
        P^{1/2} + R^{1/2} + (PR)^{1/3}
    \right)^{1+\varepsilon} \\
    \ll
    C^{1/4+\varepsilon}
    \left(
        (1+|t|)^{-d/4} + \big((1+|u|)(1+|u+t|^{n-1})\big)^{-d/8} + (1+|t|)^{-d/12}
    \right).
\end{multline*}
By \eqref{eq:fe11} we can assume that $|t| \geq |u|$ and Theorem  \ref{thm:sub1} follows easily. Similarly, we use \eqref{eq:fe11} to conclude that $Z_1(1/2, w; \psi, \psi') \ll \max_{\rho, \rho'} |Z_1(w, 1/2, \rho, \rho')|$ to conclude Theorem \ref{thm:sub}.

\section{Proofs of the Corollaries}

\begin{proof}[Proof of Corollary \ref{cor:kor2}]
This follows the pattern of Lemmas \ref{lem:ref1} and \ref{lem:ref2}, so we can be brief.  With later applications in mind, we will show the following slightly stronger bound. If $\A = \A_1\A_2^n \in \mathcal{I}(S)$ with $\A_1$ $n$-th power free, we will show
\begin{equation}\label{eq:sharper}
    \sum_{\substack{\A = \A_1\A_2^n \in \I(S) \\ \n\A \leq N}} |L(1/2+it, \chi_{\A})|^2 (\n\A_2)^{n-2+\varepsilon}
    \ll
    \left( N (1+|t|)^{d/2}\right)^{1+\varepsilon} .
\end{equation}
Clearly this implies Corollary \ref{cor:kor2}, but it also implies
\begin{equation}\label{eq:neededlater}
  \sum_{\substack{\A \in \mathcal{I}(S)\\ \n\A \leq N}} |L^{\ast}(1/2, \textbf{1}, \chi_{\A})|^2 \ll N^{1+\varepsilon},
\end{equation}
cf.\ \eqref{eq:defL*} and \eqref{eq:defa}.

A standard approximate functional equation for the Hecke $L$-function
$L(1/2 + it, \chi_\A)$ as in \cite[Theorem 5.3/Proposition 5.4]{IK} gives
\[
    L(1/2 + it, \chi_\A)
    \ll
    \Bigl|\sum_{\B \subseteq \oh } \frac{\chi_\A(\B)}{\n\B^{1/2 \pm it}}
    V\bigg(\frac{\n\B}{\text{cond}(\chi_\A)^{1/2} (1+|t|)^{d/2}}\bigg)\Bigr|
\]
where $V$ depends on $t$ and is rapidly decaying, uniformly in $t$. Similarly as in Lemma \ref{lem:ref1} and \ref{lem:ref2} we decompose uniquely $\A = \A_0\A_1\A_2^n$ where $\A_0$ is squarefree, $\A_1$ is squarefull, but $n$-th power free, and $(\A_0, \A_1) = 1$;
and $\B = \B_0\B_1\B_2$ where $\B_1$ is squarefree, $\B_2$ is squarefull, $(\B_1, \B_2) = (\B_1\B_2, S) = 1$ and $\B_0$ is only composed of prime ideals in $S$. We note that $\sum_{\B_0, \B_2} (\n\B_0\B_2)^{-1/2-\varepsilon}$ is convergent.  If $\n\A \leq N$, we can cut the $\B$-sum at
\begin{displaymath}
    \n\B_1 \leq B:=  \left(\Bigl(\frac{N}{\n\A_2^n}\Bigr)^{1/2}(1+|t|)^{d/2}\right)^{1+\varepsilon}
 \end{displaymath}
     with a negligible error. We write $V$ as an inverse Mellin transform, shift the contour to the $\varepsilon$-line and get
\begin{displaymath}
\begin{split}
   \left|L(1/2 + it, \chi_\A)\right|^2
   & \ll
    (NB)^{\varepsilon} \int_{-(NB)^{\varepsilon}}^{(NB)^{\varepsilon}}
        \Bigl|
            \sum_{\substack{\B \subseteq\oh   \\  \n\B_1 \leq B}} \frac{\chi_\A(\B)}{\n\B^{1/2+\varepsilon \pm it+iy}}
        \Bigr|^2 dy \\
&\ll  (NB)^{\varepsilon} \int_{-(NB)^{\varepsilon}}^{(NB)^{\varepsilon}}
        \Bigl|
            \underset{\substack{\B_1 \in \mathcal{I}(S)   \\  \n\B_1 \leq B}}{\left.\sum\right.^{\ast}} \frac{\chi_\A(\B_1)}{\n\B_1^{1/2+\varepsilon \pm it+iy}}
        \Bigr|^2 dy.
\end{split}
\end{displaymath}
We sum this over $\A = \A_0\A_1\A_2^n$ and conclude from Theorem \ref{thm:thm3} that
\begin{displaymath}
    \sum_{\substack{\A  \in \I(S) \\ \n\A \leq N}} |L(1/2+it, \chi_{\A})|^2 (\n\A_2)^{n-2+\varepsilon} \ll (NB)^{\varepsilon} \sum_{\n\A_1\A_2^n \ll N }   
    \left(\frac{N}{\n\A_1\A_2^n} + B +\Bigl(\frac{NB}{\n\A_1\A_2^n}\Bigr)^{2/3}\right) (\n\A_2)^{n-2}.
\end{displaymath}
The desired bound \eqref{eq:sharper} follows now easily.
\end{proof}


\begin{proof}[Proof of Corollary \ref{cor:kor2a}]
Let $r$ be a sufficiently large real number such that the asymptotic formula \cite[Theorem 1.3]{FHL} for the first moment holds. Rcalling \eqref{eq:defL*}, an application of the Cauchy-Schwarz inequality together with  \eqref{eq:neededlater} 
gives:
\[
\begin{split}
   &\#\left\{ \A \in \I(S) \mid \n\A \leq N, \,\, L(1/2, \chi_\A) \neq 0 \right\}
     \geq
    \#\left\{\A \in \I(S) \mid \n\A \leq N, \,\, L^*(1/2, \textbf{1}, \chi_\A) \neq 0\right\}  \\
    &\phantom{HHHHH}\geq
    \Bigl(
        \sum_{\substack{\A \in \I(S)  \\  \n\A \leq N}}
            L^*(1/2, \textbf{1}, \chi_\A) \Bigl(1 - \frac{\n\A}{N}\Bigr)^r
    \Bigr)^2
    \Bigl(
        \sum_{\substack{\A \in \I(S)  \\  \n\A \leq N}}
            |L^*(1/2, \textbf{1}, \chi_\A)|^2 \Bigl(1 - \frac{\n\A}{N}\Bigr)^{2r}
    \Bigr)^{-1} \\
    &\phantom{HHHHH}\gg
    N^{1-\varepsilon}.
    \hfill{} \qedhere
\end{split}
\]
\end{proof}

\begin{proof}[Proof of Corollary \ref{cor:kor3}]
Consider the zero-detecting Dirichlet polynomial
\[
    M_{X}(s, \chi_\A) =
    \sum_{\substack{\B \in \I(S)  \\  \n\B \leq X}} \frac{\mu(\B) \chi_\A(\B)}{\n\B^s}
\]
($X \geq 1$). 
Note that we may replace $L(s, \chi_\A)$ by $L_S(s, \chi_\A)$ in Corollary \ref{cor:kor3} as their zeros coincide. We now closely follow \cite[Section 11]{HB}. Choosing another parameter $Y \geq X$, we see that
\begin{multline}
\label{eq:zerodens}
    \sumstar{   \n\A \leq N} N(\sigma, T, \A)
    \ll
    (NT)^{\varepsilon} T
    \bigg(
        Y^{1/2 - \sigma} \sup_{|u| \ll T}
        \sumstar{ \n\A \leq N}
            |L_S(1/2 + iu, \chi_\A) M_X(1/2 + iu,  \chi_\A)| \\
        + Y^{\varepsilon} \sup_{X \leq U \leq Y^{1+\varepsilon}} U^{-2\sigma}  \sumstar{   \n\A \leq N}
            \Bigl| \sum_{\substack{\B \in \I(S)  \\  U \leq \n\B \leq 2U}} c_{\B} \chi_\A(\B) \Bigr|^2
    \bigg)
\end{multline}
for a certain sequence of complex numbers $c_{\B} \ll \n\B^{\varepsilon}$, cf.\ the second display on p.\ 273 and the third display on p.\ 274 of \cite{HB}. By Cauchy-Schwarz, Corollary \ref{cor:kor2} and Theorem \ref{thm:thm3}, the first term is
\begin{equation}
\label{eq:term2}
    \ll (NTX)^{\varepsilon} T  Y^{1/2-\sigma} (N T^{d/2})^{1/2}
    \big(X+ N + (XN)^{2/3}\big)^{1/2}.
\end{equation}
For the second term we write $\B = \B_1\B_2$ where $\B_1$ is squarefree and $\B_2$ is squarefull and $(\B_1, \B_2) = 1$; by Cauchy-Schwarz we can bound the double sum by
\[
    \Bigl(
        \sum_{\substack{\B_2 \leq 2U\\ \B_2 \text{ squarefull}}}
        \Bigl(
            \sumstar{   \n\A \leq N}
            \Bigl|
                \sumstar{\substack{ U \leq \n\B_1 \B_2 \leq 2U\\ (\B_1, \B_2) = 1}} c_{\B_1 \B_2} \chi_\A(\B_1)
            \Bigr|^2
        \Bigr)^{1/2}
    \Bigr)^2.
\]
We split the $\A$ and $\B_1$ sums into residue classes modulo $R_{\mathfrak{c}}$ and apply reciprocity. Then we can use Theorem \ref{thm:thm3} to see that the previous display is at most
\[
    \ll (NU)^{\varepsilon}
    \Bigl(\sum_{\substack{\B_2 \leq 2U  \\  \B_2 \text{ squarefull}}}
    \Bigl((N + U + (NU)^{2/3}) \frac{U}{\n\B_2}\Bigr)^{1/2} \Bigr)^2
    \ll (NU)^{\varepsilon} \big(N+U+(NU)^{2/3}\big) U.
\]
Hence the second term in \eqref{eq:zerodens} is at most
\begin{equation}
\label{eq:term1}
    \ll (NY)^{\varepsilon}
    (NX^{1-2\sigma} + N^{2/3} X^{5/3-2\sigma} + N^{2/3} Y^{5/3-2\sigma} +
        Y^{2-2\sigma})
\end{equation}
for $1/2 \leq \sigma \leq 1$ and $1 \leq X \leq Y$. Note that of the two terms
$N^{2/3} X^{5/3-2\sigma} + N^{2/3} Y^{5/3-2\sigma}$
only one contributes to the sum: the latter if $\sigma < 5/6$ and the former if $\sigma \geq 5/6$.   Combining \eqref{eq:zerodens}, \eqref{eq:term2} and \eqref{eq:term1} we find
\begin{multline*}
    \sumstar{  \n\A \leq N} N(\sigma, T, \A)
    \ll
    (NTY)^{\varepsilon} T
    \Bigl(
        Y^{1/2-\sigma} (N T^{d/2})^{1/2} (X+ N + (XN)^{2/3})^{1/2} \\
            +  NX^{1-2\sigma} + N^{2/3} X^{5/3-2\sigma} + N^{2/3} Y^{5/3-2\sigma} + Y^{2-2\sigma}
    \Bigr).
\end{multline*}
It remains to optimize $X$ and $Y$ in terms of $T$, $N$ and $\sigma$. For $1/2 < \sigma \leq 5/6$ we can choose
\[
    X = N^{1/2}, \quad Y = N^{\frac{2}{7 - 6\sigma}} T^{\frac{d}{6-4\sigma}}
\]
to get
\begin{equation}
\label{eq:density1}
    \sumstar{  \n\A \leq N} N(\sigma, T, \A)
    \ll
    (NT)^{\varepsilon}
    N^{\frac{8(1-\sigma)}{7-6\sigma}}
    T^{d \frac{1-\sigma}{3-2\sigma}+1},
    \quad 1/2 < \sigma \leq 5/6.
\end{equation}
This is nontrivial in the $N$-aspect for $\sigma > 1/2$. In the range $5/6 \leq \sigma \leq 1$ we can choose
\[
    X = N^{1/2} + N^{\frac{6\sigma-4}{24\sigma - 12\sigma^2 - 11}},
    \quad
    Y =  N^{\frac{10\sigma-7}{24\sigma - 12\sigma^2 - 11}} T^{\frac{d}{6-4\sigma}}.
\]
This balances $Y^{1/2-\sigma}(N T^{d/2})^{1/2}  (XN)^{1/3} = Y^{2-2\sigma} = N^{2/3} Y^{5/3-2\sigma}$ in the $N$-aspect, while the choice for the $T$-exponent was just made for simplicity and is not optimal. Note that $X \leq Y$ and $N^{1/2} \leq X \leq N^2$, so that $X+Q+(XQ)^{2/3} \asymp (XQ)^{2/3}$. Hence we obtain
\begin{equation}
\label{eq:density2}
    \sumstar{   \n\A \leq N} N(\sigma, T, \A)
    \ll
    (NT)^{\varepsilon} Q^{\frac{2(10\sigma-7)(1-\sigma)}{24\sigma - 12\sigma^2 - 11}} T^{d \frac{1-\sigma}{3-2\sigma} + 1}
\end{equation}
for $ 5/6 < \sigma \leq 1$. This completes the proof.
\end{proof}

\section*{Appendix I}

We briefly sketch a convexity argument leading to \eqref{eq:convexity}.
To this end we assume
\begin{equation}
    \label{eq:assump}
    \sum_{\n\A \leq X}  |D(w, \A)|
    \ll
    X^{1+\varepsilon} |w|^{(n-1) d/4+\varepsilon}, \quad \Re w \geq 1/2,
\end{equation}
i.e.\ convexity in the $w$-aspect combined with Lindel\"of in the conductor aspect on average.
This is not a trivial fact, but can be proved using the same approach as in Corollary \ref{cor:kor2}. The bound \eqref{eq:assump} implies
\[
    Z_2(s, w) \ll |w|^{(n-1)d/4+\varepsilon} ,
    \quad \Re s = 1+\varepsilon, \quad \Re w = 1/2.
\]
Applying the first  functional equation in \eqref{eq:f1+}, we find after a short calculation
\[
    Z_1(s, w) \ll |s|^{d/2+\varepsilon}|s+w|^{(n-1)d/4+\varepsilon},
    \quad \Re s = -\varepsilon, \quad \Re w = 1+\varepsilon.
\]
By \eqref{eq:f1} we conclude the same bound with $s$ and $w$ interchanged. Interpolating between the latter two bounds gives the desired convexity bound \eqref{eq:convexity}. Note that interpolating naively between $(\Re s, \Re w) = (1+\varepsilon, 1+\varepsilon)$ and $(-\varepsilon, -\varepsilon)$ via \eqref{eq:f1++} one can only prove
${Z_1(s, w) \ll} {(|s| |s+w|^{2(n-1)}|w|)^{d/4}}$ for $\Re s = \Re w = 1/2$, which shows that the notion of convexity is much more subtle in the context of multiple Dirichlet series.

\section*{Appendix II}

In order to keep this paper self-contained, we present here the precise details of the construction of the character $\chi_{\A}$. It slightly simplifies the presentation if we assume $n \geq 3$ which we may for the purpose of this paper.

Given $K$ and $n$, we fix once and for all a finite set of finite places $S$ of $K$ containing all divisors of $n$  such that the ring $\mathcal{O}_S$ of $S$-integers has class number 1.  We also fix once and for all an ideal $\mathfrak{c} = \prod_{\mathfrak{p} \in S_f} \mathfrak{p}^{e_{\mathfrak{p}}}$ where $e_{\mathfrak{p}} \in \Bbb{N}$ is sufficiently large to ensure that $\text{ord}_{\mathfrak{p}}(a-1) \geq e_{\mathfrak{p}}$ for $a \in K_{\mathfrak{p}}$ implies that $a \in K_{\mathfrak{p}}$ is an $n$-th power. If $\mathfrak{p} \nmid n$, we can choose $e_{\mathfrak{p}} = 1$. Here, as usual, $K_{\mathfrak{p}}$ denotes the completion of $K$ at $\mathfrak{p}$.  Let $H_{\mathfrak{c}}$ be the ray class group modulo $\mathfrak{c}$ and $R_{\mathfrak{c}} := H_{\mathfrak{c}} \otimes \Bbb{Z}/n\Bbb{Z}$.  Write $R_{\mathfrak{c}}$ as a product of cyclic groups and choose a generator for each. Let $\mathcal{E}_0 \subseteq \mathcal{I}(S)$ be a complete set of ideals representing these generators, and choose for each $E_0 \in \mathcal{E}_0$ some $m_{E_0} \in K^{\times}$ such that $(m_{E_0}) = E_0$ as $\mathcal{O}_S$-ideals. Let $\mathcal{E} \subseteq \mathcal{I}(S)$ be a complete set of representatives of $R_{\frakc}$.
Write each $E\in \mathcal{E}$ as $\prod E_0^{n_{E_0}}$ with $n_{E_0} \geq 0$ and let $m_E = \prod m_{E_0}^{n_{E_0}} \in K^{\times}$. Without loss of generality assume $(1) \in \mathcal{E}$ and $m_{(1)} = 1$.

For coprime $\mathfrak{a}, \mathfrak{b} \in \mathcal{I}(S)$ we decompose $\mathfrak{a} = (x) E \mathfrak{g}^n$ with $x \equiv 1$ (mod $\mathfrak{c}$), $E \in \mathcal{E}$ and $(\mathfrak{g}, \mathfrak{b}) = 1$. Then we define $\chi_{\mathfrak{a}}(\mathfrak{b}) := (xm_E/\mathfrak{b})$ where the latter is the usual $n$-th power residue symbol given by the Artin map. We recall that this symbol has the property that $(a/\mathfrak{b}) = 1$ if $a \equiv 1$ (mod $\mathfrak{b}$).

One can show that $\chi_{\mathfrak{a}}$ is well-defined \cite{FF, FHL} and gives a Hecke character modulo $\mathfrak{c} \mathfrak{a}$ (see \cite{GL}). We then agree to re-define $\chi_{\mathfrak{a}}$ as the underlying primitive Hecke character.  The construction of this character is not canonical, as it depends on all the choice made above. However, the vector space of double Dirichlet series spanned by $\{Z_1(s, w, \psi, \psi') \mid \psi, \psi' \in R_{\mathfrak{c}}\}$ is independent of these choices (\cite[Proposition 1.2]{FF}).


\end{document}